\documentclass[10pt]{amsart}

\usepackage{amsmath,amssymb,amsthm,graphicx,pstricks,comment,amscd,amsfonts,latexsym}
\usepackage[all]{xy}

\usepackage{lscape}
\usepackage{hyperref}
\usepackage{tikz}

\usepackage{epstopdf}

\numberwithin{equation}{section}

\newtheorem{theorem}{Theorem}[section]

\newtheorem{lemma}[theorem]{Lemma}
\newtheorem{corollary}[theorem]{Corollary}
\newtheorem{proposition}[theorem]{Proposition}
\theoremstyle{remark}
\newtheorem{remark}[theorem]{Remark}
\newtheorem{example}[theorem]{Example}
\theoremstyle{definition}
\newtheorem{definition}[theorem]{Definition}

\newcommand{\flip}{\mathfrak{f}}

\newcommand{\suchthat}{ \ | \ }
\newcommand{\myid}{1\hspace{-0.15cm}1}
\newcommand{\ZZ}{\mathbb{Z}}
\newcommand{\CC}{\mathbb{C}}

\newcommand{\pathalg}[1]{\CC\langle#1\rangle}
\newcommand{\RA}[1]{\CC\langle\hspace{-0.05cm}\langle #1\rangle\hspace{-0.05cm}\rangle}

\newcommand{\initial}{\operatorname{init}}
\newcommand{\terminal}{\operatorname{term}}

\newcommand{\surface}{\Sigma}
\newcommand{\marked}{\mathbb{M}}
\newcommand{\punct}{\mathbb{P}}
\newcommand{\surf}{(\surface,\marked)}
\newcommand{\surfaceprime}{\surface'}
\newcommand{\markedprime}{\marked'}
\newcommand{\surfprime}{(\surfaceprime,\markedprime)}
\newcommand{\triangtau}{\tau}
\newcommand{\Qtau}{Q(\triangtau)}
\newcommand{\unredQtau}{\widehat{Q}(\triangtau)}
\newcommand{\Stau}{S(\triangtau)}
\newcommand{\unredStau}{\widehat{S}(\triangtau)}
\newcommand{\QStau}{(\Qtau,\Stau)}

\newcommand{\triangsigma}{\sigma}

\newcommand{\bipgraph}[1]{G(#1)}
\newcommand{\whiteset}{\mathbb{W}}
\newcommand{\blackset}{\mathbb{B}}
\newcommand{\greyset}{\mathbb{G}}
\newcommand{\triang}{\triangle}
\newcommand{\surftriang}{(\surface_\triang,\marked_\triang)}
\newcommand{\surfpunct}[1]{(\surface_{#1},\marked_{#1})}
\newcommand{\head}{h}
\newcommand{\tail}{t}
\newcommand{\Hom}{\operatorname{Hom}}
\newcommand{\gr}{\operatorname{gr}}

\newcommand{\End}{\operatorname{End}}

\newcommand{\unredCtau}[1]{\widehat{C}_{#1}(\tau)}
\newcommand{\Ctau}[1]{C_{#1}(\tau)}

\begin{document}

\title{Derived invariants for surface cut algebras II: The punctured case}
\author[Claire Amiot, Daniel Labardini-Fragoso and Pierre-Guy Plamondon]{Claire Amiot \and Daniel Labardini-Fragoso\and  Pierre-Guy Plamondon}
\address{Institut Fourier, Universit\'{e} Joseph Fourier, Grenoble, France.}
\email{claire.amiot@ujf-grenoble.fr}
\address{Instituto de Matem\'{a}ticas, Universidad Nacional Aut\'{o}noma de M\'{e}xico.}
\email{labardini@matem.unam.mx}
\address{Laboratoire de Math\'{e}matique d'Orsay, Univ. Paris-Sud, CNRS, Univ. Paris-Saclay, 91405 Orsay, France.}
\email{pierre-guy.plamondon@math.u-psud.fr}
\thanks{The first and third authors were partially supported by the French ANR grant SC3A (ANR-15-CE40-0004-01). The second author was supported by the grants \emph{PAPIIT-IA102215} and \emph{CONACyT-238754}.}
\maketitle

\begin{abstract}
For each algebra of global dimension 2 arising from the quiver with potential associated to a triangulation of an unpunctured surface, Amiot-Grimeland have defined an integer-valued function on the first singular homology group of the surface, and have proved that two such algebras of global dimension 2 are derived equivalent precisely when there exists an automorphism of the surface that makes their associated functions coincide.

In the present paper we generalize the constructions and results of Amiot-Grimeland to the setting of arbitrarily punctured surfaces. As an application, we show that there always is a derived equivalence between any two algebras of global dimension 2 arising from the quivers with potential of (valency $\geq 2$) triangulations of arbitrarily punctured polygons.

While in the unpunctured case the quiver with potential of any triangulation admits cuts yielding algebras of global dimension at most 2, in the case of punctured surfaces the QPs of some triangulations do not admit cuts, and even when they do, the global dimension of the corresponding degree-0 algebra may exceed 2. In this paper we give a combinatorial characterization of each of these two situations.
\end{abstract}

\tableofcontents


\section{Introduction}


The aim of this paper is to study derived equivalences for a class of finite-dimensional algebras defined from triangulations of punctured surfaces with boundary.

The algebras we consider in this paper arose from the connections between Fomin and Zelevinsky's cluster algebras \cite{Fomin-Zelevinsky} and two subjects: triangulations of Riemann surfaces and representation theory of algebras.

On one hand, given a triangulation of an orientable surface (possibly with boundary) with marked points, one can consider the adjacency quiver of this triangulation.  It was proved in \cite{FST} that the combinatorics of the flips of the triangulation are closely related to the combinatorics of mutation in the cluster algebra associated to the quiver.  

On the other hand, it was discovered in \cite{Caldero-Chapoton} that the representations of a quiver allow us to explicitly recover the generators of cluster algebras.  The most general results are obtained by considering quivers with potentials (as introduced in \cite{DWZ1}); the combinatorics of the associated cluster algebra is then encoded in a triangulated category, the generalized cluster category, defined by the first author in \cite{Amiot-gldim2}.

These two approaches to cluster algebras can be linked in the following way: to the quiver of a triangulated surface, the second author associated a potential in \cite{LF1} (see also \cite{Caldero-Chapoton-Schiffler} and \cite{ABCP}), and proved, in broad terms, that the combinatorics of the two approaches agree.  

The process of mutation in the cluster category bears many similarities to tilting theory; in particular, it was proved in \cite{Assem-Brustle-Schiffler} that to any tilted algebra given by a quiver with relations, one can associate a cluster-tilted algebra defined by a quiver with potential; the new quiver is obtained from the original one by adding arrows.

Surface cut algebras (or surface algebras), introduced in \cite{DavidR-Schiffler}, are defined by the inverse process of removing arrows from a quiver with potential arising from a triangulation of a surface.  More precisely, let $(Q(\tau), S(\tau))$ be the quiver with potential associated to a triangulation $\tau$ of a marked surface $(\Sigma, \mathbb{M})$.  Let $d$ be a degree map assigning degree $0$ or $1$ to each arrow of $Q(\tau)$, in such a way that the potential $S(\tau)$ is homogeneous of degree $1$.  Such a degree map is called a cut.  The \emph{surface cut algebra} $\Lambda(\tau,d)$ is the degree zero subalgebra of the Jacobian algebra of $(Q(\tau),S(\tau))$, provided it has global dimension $2$.  Equivalently, the cut algebra is obtained from the Jacobian algebra by quotienting by the ideal generated by all arrows of degree $1$ (in other words, by ``removing'' these arrows), hence the similarity with the process in \cite{Assem-Brustle-Schiffler} where arrows are added to tilted algebras.

Since tilted algebras of a given type are derived equivalent, it becomes natural to ask whether surface cut algebras arising from the same marked surface are also derived equivalent.  Our main result characterizes the cases where this is true.

\begin{theorem}[Theorem \ref{theorem cluster cat}]
Let $\surf$ be a surface with non-empty boundary. 
 Let $\Lambda=\Lambda(\tau,d)$ and $\Lambda'=\Lambda(\tau', d')$ be surface cut algebras associated with valency $\geq 2$-triangulations $\tau$ and $\tau'$ of $\surf$, with respective admissible cuts $d$ and $d'$. Then the following statements are equivalent:
\begin{enumerate}
\item there is a triangle equivalence $\mathcal{D}^b(\Lambda)\cong\mathcal{D}^b(\Lambda')$;
\item there exists an orientation preserving homeomorphism $\Phi:\Sigma\to \Sigma$ with $\Phi(\mathbb M)=\mathbb M$ such that for any closed curve $\gamma$, $d(\overline{\gamma}^\tau)=d'(\overline{\Phi(\gamma)}^{\tau'})$.
\end{enumerate}
\end{theorem}

This generalizes results obtained in \cite{David-Roesler} (for unpunctured spheres and $\tau=\tau'$) and in \cite{Amiot-Grimeland} (for unpunctured surfaces in general).

The proof of the main theorem relies heavily on the cluster category defined in \cite{Amiot-gldim2} for algebras of global dimension $2$.  For this reason, the assumption that the surface cut algebras have global dimension $2$ is essential.  While not all algebras constructed in this way have this property, we prove the following

\begin{proposition}[Corollary \ref{cor:existence-surface-algebra}]
  Any marked surface with non-empty boundary admits a triangulation and a cut such that the resulting cut algebra has global dimension $2$.
\end{proposition}

Finally, we show that triangulations of punctured surfaces do not necessarily admit admissible cuts, which contrasts with the situation in the unpunctured case.

\begin{proposition}[Corollary \ref{coro::nocut} and Proposition \ref{prop::nocut}]
$\phantom{x}$
  \begin{itemize}
    \item If $(\Sigma, \mathbb{M})$ has empty boundary and is not a sphere with less than $5$ punctures, and if $\tau$ is an ideal triangulation of $(\Sigma, \mathbb{M})$ for which every puncture has valency at least $3$, then $(Q(\tau), S(\tau))$ admits no admissible cuts.
    \item If $(\Sigma, \mathbb{M})$ has non-empty boundary and at least one puncture, then there exist triangulations of $(\Sigma, \mathbb{M})$ whose associated quivers with potential do not admit admissible cuts.
  \end{itemize}
\end{proposition}

The paper is organized as follows.  In Section \ref{section2}, we recall the different notions of gradings and graded equivalences of algebras that we will need.  We also discuss mutation of graded quivers.  In Section \ref{section3}, we define a complex whose homology is isomorphic to that of the surface, and which will allow us to evaluate gradings at curves on the surface.  We also discuss graded flips of triangulations.  In Section \ref{section4}, we prove a certain compatibility between graded mutation of quivers with potentials and graded flips of triangulations. In Section \ref{section5}, we state and prove our main theorem on derived equivalences of surface cut algebras.  Section \ref{section6} is devoted to the existence (or not) of admissible cuts, and Section \ref{section7} to the computation of the global dimension of cut algebras.


\section{Gradings and mutation}\label{section2}


In this section we recall the notions of graded equivalence of graded algebras, graded mutation of graded quivers, and graded mutations of graded quivers with potential. Graded mutations of graded quivers will be the combinatorial counterpart of graded mutations of graded quivers with potential. We will see that if the quiver $Q$ of a graded quiver $(Q,d)$ admits a non-degenerate potential which is homogeneous of degree 1 with respect to $d$, then it is possible to apply to $(Q,d)$ any finite sequence of graded mutations of graded quivers without the need of mutating potentials along the way.

\subsection{Graded equivalence}

\begin{definition}\label{def:grading-on-a-quiver} Let $Q=(Q_0,Q_1,\head,\tail)$ be a quiver. A \emph{grading on $Q$} is any function $d:Q_1\rightarrow\ZZ$.
\end{definition}

Let $Q$ be a quiver and $d$ be a grading on $Q$. As it is customary, for a positive-length path $a_\ell\ldots a_1$ on $Q$ we define $d(a_\ell\ldots a_1)=d(a_\ell)+\ldots+d(a_1)$, and set $d(e_i)=0$ for every length-$0$ path $e_i$. Letting $\pathalg{Q}_{\ell}$ (resp. $\RA{Q}_{\ell}$) be the $\CC$-vector subspace that consists of the $\CC$-linear combinations (resp. the possibly infinite $\CC$-linear combinations) of paths that have a given degree $\ell\in\ZZ$ with respect to $d$, we see that $d$ induces a $\ZZ$-grading on the algebras
\begin{equation}\label{eq:grading-on-path-alg-and-complete-path-alg}
\pathalg{Q} = \bigoplus_{\ell\in\ZZ}\pathalg{Q}_{\ell} \ \ \ \text{and} \ \ \ \RA{Q} = \prod_{\ell\in\ZZ}\RA{Q}_{\ell}.
\end{equation}
We will say that $d$ \emph{is a grading on} $\pathalg{Q}$ and $\RA{Q}$, despite the fact that the decomposition of $\RA{Q}$ in \eqref{eq:grading-on-path-alg-and-complete-path-alg} is as a direct product and not as a direct sum.

\begin{remark} Let $Q$ be a quiver, $d$ a grading on $Q$, and $I$ an ideal of $\pathalg{Q}$ which is homogeneous with respect to $d$. Then the topological closure $\overline{I}$ of $I$ in $\RA{Q}$ is homogenous in the sense that $\overline{I}=\prod_{\ell}(\RA{Q}_{\ell}\cap\overline{I})$; hence there are decompositions
\begin{equation}\label{eq:gradings-on-quotients-of-path-alg-and-complete-path-alg}
\pathalg{Q}/I = \bigoplus_{\ell\in\ZZ}\left(\pathalg{Q}_{\ell}/(\pathalg{Q}_{\ell}\cap I)\right) \ \ \ \text{and} \ \ \ \RA{Q}/\overline{I} = \prod_{\ell\in\ZZ}\left(\RA{Q}_{\ell}/(\RA{Q}_{\ell}\cap\overline{I})\right),
\end{equation}
and the inclusion $\pathalg{Q}\hookrightarrow\RA{Q}$ induces a degree-preserving $\CC$-algebra homomorphism $\pathalg{Q}/I\rightarrow\RA{Q}/\overline{I}$. If $I$ is an admissible ideal of $\pathalg{Q}$ and this algebra homomorphism happens to be an isomorphism, then $\RA{Q}_{\ell}/(\RA{Q}_{\ell}\cap\overline{I})=0$ for $|\ell|\gg0$ and $\left(\pathalg{Q}_{\ell}/(\pathalg{Q}_{\ell}\cap I)\right)=\RA{Q}_{\ell}/(\RA{Q}_{\ell}\cap\overline{I})$ for all $\ell\in\ZZ$, hence the decomposition of $\RA{Q}$ given in \ref{eq:gradings-on-quotients-of-path-alg-and-complete-path-alg} becomes
\begin{equation*}
\RA{Q}/\overline{I} = \bigoplus_{\ell\in\ZZ}\left(\RA{Q}_{\ell}/(\RA{Q}_{\ell}\cap\overline{I})\right) = \bigoplus_{\ell\in\ZZ}\left(\pathalg{Q}_{\ell}/(\pathalg{Q}_{\ell}\cap\overline{I})\right).
\end{equation*}
Let $d'$ be a grading on $Q'$ and $I'$ be an ideal of $\pathalg{Q'}$ which is homogenous with respect to $d'$. In a slight abuse, we will say that $\varphi:(\RA{Q}/\overline{I},d)\to (\RA{Q'}/\overline{I'},d')$ is a morphism of graded algebras if $\varphi$ is an algebra morphism which sends homogenous elements of degree $\ell$ with respect to $d$ to homogenous elements of degree $\ell$ with respect to $d'$.
\end{remark}

\begin{definition}\label{def:graded-equiv-of-algebras} Let $Q=(Q_0,Q_1,\head,\tail)$ and $Q'=(Q'_0,Q'_1,\head',\tail')$ be quivers, $d:Q_1\rightarrow\ZZ$ (resp. $d':Q'_1\rightarrow\ZZ$) be a grading on $Q$ (resp. $Q'$), and $I$ (resp. $I'$) be an admissible ideal of $\pathalg{Q}$ (resp. $\pathalg{Q'}$), homogeneous with respect to the grading $d$ (resp. $d'$), so that $d$ (resp. $d'$) induces a grading on the quotient algebra $\Lambda=\pathalg{Q}/I$ (resp. $\Lambda'=\pathalg{Q'}/I'$), grading which we will still denote $d$ (resp. $d'$). The graded algebras $(\Lambda,d)$ and $(\Lambda',d')$ are said to be \emph{graded equivalent} if there exists a tuple of integers  $(r_i)_{i\in Q_0}$ and an isomorphism of graded algebras
$$
\Lambda'\underset{\ZZ}{\cong}\bigoplus_{p\in\ZZ}\Hom_{\gr \Lambda}\left(\bigoplus_{i\in Q_0}P_i (r_i),\bigoplus_{j\in Q_0}P_j( r_j+p)\right),
$$
where $P_i(r_i)$ is the indecomposable graded projective $\Lambda$-module associated to the vertex $i\in Q_0$ shifted by $r_i$. If $Q=Q'$ and $I=I'$ and if the isomorphism of algebras (forgetting the grading) is the identity, then we say that $\Lambda$ and $\Lambda'$ are \emph{graded equivalent via the identity}.
\end{definition}

Note that by \cite[Theorem 5.3]{GG}, $(\Lambda,d)$ and $(\Lambda',d')$ are graded equivalent if and only if there is an equivalence $\gr \Lambda\simeq \gr \Lambda'$ between the categories of finitely generated graded modules.

\begin{definition}\label{def:equivalence-of-quiver-gradings} Let $Q=(Q_0,Q_1,\head,\tail)$ be a quiver. Two gradings $d_1$ and $d_2$ on $Q$ are \emph{equivalent} if there exists a function $r:Q_0\rightarrow\ZZ$ such that $d_1(a)=d_2(a)+r(\head(a))-r(\tail(a))$ for every $a\in Q_1$.
\end{definition}

\begin{remark}\label{rema::graded-equiv} 
Two graded algebras $(\Lambda, d)$ and $(\Lambda', d')$ are graded equivalent if and only if there is an isomorphism of quivers $\phi:Q\to Q'$ of quivers sending $I$ to $I'$ and such that $d$ and $d'\circ \phi$ are equivalent gradings on $Q$.
\end{remark}


\subsection{Graded mutations of graded quivers}

In \cite[\S 6.3]{AO-cl-equiv-der-equiv}, Amiot-Oppermann refine Derksen-Weyman-Zelevinsky's definition \cite{DWZ1} of mutations of quivers with potential by defining a graded version. A careful look at the referred subsection of \cite{AO-cl-equiv-der-equiv} shows that if a 2-acyclic quiver $Q$ and a grading $d$ on it are fixed, then the sole knowledge of existence of a non-degenerate potential $S\in\RA{Q}$ of degree $1$ with respect to $d$ allows to perform any sequence of graded mutations on the pair $(Q,d)$ purely combinatorially, without having to apply the corresponding mutations at the level of potentials. The aim of this subsection is to explicitly describe graded quiver mutations as a combinatorial counterpart of graded mutations of graded QPs which should be thought to be analogous to how ordinary quiver mutations are the combinatorial counterpart of Derksen-Weyman-Zelevinsky's mutations of QPs.

\begin{definition}\cite{AO-cl-equiv-der-equiv}\label{def:mutation-of-graded-quiver}
Let $(Q,d)$ be a pair consisting of a 2-acyclic quiver and a grading $d:Q_1\rightarrow\ZZ$. For $k\in Q_0$, let $\widetilde{\mu}_k^L(Q,d):=(\widetilde{\mu}_k^L(Q),\widetilde{\mu}_k^L(d))$ be defined as follows:
\begin{enumerate}
\item $\widetilde{\mu}_k^L(Q)$ is the quiver obtained from $Q$ by applying the following two-step procedure:
\begin{itemize}
\item For each pair of arrows $j\overset{a}{\rightarrow}k\overset{b}{\rightarrow}i$, add a composite arrow $j\overset{[ab]}{\rightarrow}i$;
\item replace every arrow $a$ incident to $k$ with an arrow $a^*$ going in the opposite direction;
\end{itemize}
\item $\widetilde{\mu}_k^L(d)$ is the grading $\widetilde{d}$ on $\widetilde{\mu}_k^L(Q)$ defined by
\begin{itemize}
\item $\widetilde{d}(\alpha)=d(\alpha)$ for any arrow $\alpha\in Q_1\cap \widetilde{\mu}_k^L(Q)_1$;
\item $\widetilde{d}([ab])=d(a)+d(b)$ for every composite arrow $[ab]$ of $\widetilde{\mu}_k^L(Q)$;
\item $\widetilde{d}(a^*)=-d(a)$  for every arrow $a$ of $Q$ satisfying $t(a)=k$;
\item $\widetilde{d}(b^*)=1-d(b)$  for every arrow $b$ of $Q$ with the property that $h(b)=k$.
\end{itemize}
\end{enumerate}
If the quiver obtained from $\widetilde{\mu}_k^L(Q)$ by removing a maximal collection of disjoint degree-1 cycles of length 2 happens to be 2-acyclic, then we denote it $\mu_k^L(Q)$, and call the pair $\mu_k^L(Q,d):=(\mu_k^L(Q),\mu_k^L(d))$ the \emph{left graded mutation} of $(Q,d)$ with respect to $k$, where 
$\mu_k^L(d)$ is the grading on $\mu_k^L(Q)$ obtained by restricting $\widetilde{\mu}_k^L(d)$ to $\mu_k^L(Q)_1$.
\end{definition}

\begin{remark}\label{rmk:gradedmut} Note that if $\mu_k^L(Q)$ is defined, then:
\begin{enumerate}\item Up to isomorphism of graded quivers, the pair $\mu_k^L(Q,d):=(\mu_k^L(Q),\mu_k^L(d))$ is independent of the maximal collection of disjoint degree-1 cycles of length 2 removed from $\widetilde{\mu}_k^L(Q)$; 
\item $\mu_k^L(Q)$ is isomorphic to the quiver obtained from $Q$ by ordinary quiver mutation, therefore we denote it by $\mu_k(Q)$ in the rest of the paper; and
\item an easy computation shows that the gradings $d$ and $\mu_k^L\circ\mu_k^L(d)$ are equivalent as gradings on $Q=\mu_k\circ\mu_k(Q)$.
\end{enumerate}
\end{remark}

\begin{definition}\label{def:gradeQP} A \emph{graded quiver with potential} $(Q,S,d)$ (graded QP for short) is a quiver with potential $(Q,S)$ in the sense of \cite{DWZ1} together with a grading on $Q$ making $S$ homogenous of degree $1$. A \emph{graded  right equivalence} $\varphi:(Q,S,d)\to (Q',S',d')$ between two such graded QP with same vertex set $Q_0$ is a right equivalence $\varphi: (Q,S)\to (Q',S')$ in the sense \cite{DWZ1} which is an isomorphism of graded algebras $(\RA{Q},d)\to(\RA{Q'},d')$. We denote it by $(Q,S,d)\underset{\mathbb Z}{\cong}(Q',S',d')$. 
If $(Q,S,d)$ is a graded QP, the associated Jacobian algebra $\mathcal{P}(Q,S,d)$ (see \cite{DWZ1} for definition), inherits a natural grading.
\end{definition}

\begin{proposition}\cite{AO-cl-equiv-der-equiv}\label{prop:mutation-of-graded-QPs} 
Let $Q$ be a $2$-acyclic quiver and $(Q,S,d)$ be a graded QP. Then for any vertex $k\in Q_0$: 
\begin{enumerate}
\item the potential $\widetilde{\mu}_k(S)$ is homogeneous of degree 1 with respect to $\widetilde{\mu}_k^L(d)$;
\item\label{item:graded-QP-reduction-and-mutation} there exist potentials $W_{\operatorname{red}}$ and $W_{\operatorname{triv}}$ and  a graded right equivalence 
$$ \varphi: (\widetilde{\mu}^L_k(Q),\widetilde{\mu}_k(S),\widetilde{\mu}_k^L(d))\longrightarrow (\mu_k(Q),W_{\operatorname{red}},\widetilde{\mu}_k^L(d)_{|_{\mu_k(Q)}})\oplus (C,W_{\operatorname{triv}},\widetilde{\mu}_k^L(d)_{|_C}),$$  where $C$ is the subquiver of $\widetilde{\mu}_k^L(Q)$ such that $\widetilde{\mu}_k^L(Q)=\mu_k(Q)\oplus C$, and where  $(C,W_{\operatorname{triv}})$ is a trivial QP.

\end{enumerate}

\end{proposition}

\begin{definition} We denote by $\mu_k^L(Q,S,d)$ the graded QP $(\mu_k(Q),W_{\operatorname{red}},\widetilde{\mu}_k^L(d)_{|_{\mu_k(Q)}})$.
\end{definition}

As a consequence, if $(Q,S,d)$ is a $2$-acyclic graded QP, and $S$ is non degenerate we have $\mu_k^L(Q,S,d)=(\mu_k(Q),\mu_k(S),\mu_k^L(d))$.

We stress the fact that, given $(Q,d)$, the sole existence of a degree-1 non-degenerate potential on $Q$ guarantees the well-definedness of the pair $\mu_{k_\ell}^L\ldots\mu_{k_1}^L(Q,d):=(\mu_{k_\ell}\ldots\mu_{k_1}(Q),\mu_{k_\ell}^L\ldots\mu_{k_1}^L(d))$ for any finite sequence $(k_1,\ldots,k_\ell)$ of vertices of $Q$.

\begin{remark} Part \eqref{item:graded-QP-reduction-and-mutation} of Proposition \ref{prop:mutation-of-graded-QPs} is obtained by a suitable refinement of Derksen-Weyman-Zelevinsky's \cite[Theorem 4.6]{DWZ1}. Actually, one has $\widetilde{\mu}_k(Q)=\widetilde{\mu}_k^L(Q)$, and the potential $W_{\operatorname{red}}$ and the right-equivalence $\varphi$ above can be obtained from $(\widetilde{\mu}_k(Q),\widetilde{\mu}_k(S))$ precisely by applying the limit process described in the proofs of \cite[Lemmas 4.7 and 4.8]{DWZ1}. So, the potential $W_{\operatorname{red}}$ (and not only its right-equivalence class) is precisely the one that Derksen-Weyman-Zelevinsky obtain when they show that a \emph{reduced part} of $(\widetilde{\mu}_k(Q),\widetilde{\mu}_k(S))$ does exist indeed.
\end{remark}

The following two facts are then easy to check.
\begin{lemma}\label{lemma:d1_d2-equiv=>muL(d1)_muL(d2)-equiv} Let $Q$ be a 2-acyclic quiver, and let $d_1$ and $d_2$ be equivalent gradings on $Q$. Any potential $S$ on $Q$ which is homogeneous with respect to $d_1$ is also homogeneous with respect to $d_2$, and $d_1(S)=d_2(S)$. Consequently, the graded Jacobian algebras $\mathcal{P}(Q,S,d_1)$ and $\mathcal{P}(Q,S,d_2)$ are graded equivalent via the identity. 

Moreover, if $Q$ admits a non-degenerate potential which is homogeneous of degree 1 with respect to $d_1$, then for any vertex $k\in Q_0$, $\mu_k^L(d_1)$ and $\mu_k^L(d_2)$ are equivalent gradings on $\mu_k^L(Q)$ and, consequently, the graded Jacobian algebras $\mathcal{P}(\mu_k^L(Q),\mu_k^L(S),\mu_k^L(d_1))$ and $\mathcal{P}(\mu_k^L(Q),\mu_k^L(S),\mu_k^L(d_2))$ are graded equivalent via the identity.
\end{lemma}


\section{Triangulations and gradings}\label{section3}


In this section we construct two chain complexes $\Ctau{\bullet}$ and $\unredCtau{\bullet}$ and a CW-complex $X_\tau$ for each valency $\geq 2$-triangulation $\tau$.  The chain complexes $\Ctau{\bullet}$ and $\unredCtau{\bullet}$ will be defined in terms of $\tau$ and the potentials $S(\tau)$ and $\unredStau$, while the CW-complex $X_\tau$ is defined purely in terms of the triangulation $\tau$. The main features of $X_\tau$, $\Ctau{\bullet}$ and $\unredCtau{\bullet}$ will be that $X_\tau$ is a strong deformation retract of $\Sigma$ and that the cellular homology of $X_\tau$ is isomorphic to the homology of $\Ctau{\bullet}$, features that will allow us to associate invariants in $H_1(\Sigma,\mathbb{Z})$ and $H^1(\Sigma,\mathbb{Z})$ to objects defined in terms of any fixed $\tau$.

We then define graded flips of graded tagged triangulations and relate them to mutations of graded quivers. We end the section by discussing graded triangulations having isomorphic graded quivers.

Our treatment follows closely the one presented in \cite[Section 2]{Amiot-Grimeland} and \cite[Section 2]{Amiot-genus1}.
Similar constructions have been considered by Broomhead \cite{Broomhead-paper} and Mozgovoy-Reineke \cite{Mozgovoy-Reineke}.

For the basic notions of algebraic topology used throughout we refer the reader to \cite{Cohen} and \cite{Cooke-Finney}.

\subsection{The chain complexes $\Ctau{\bullet}$ and $\unredCtau{\bullet}$}\label{subsection3.1}

We follow \cite{FST}.  Throughout the paper, a \emph{surface with marked points} will be a pair $\surf$, where $\Sigma$ is a compact Riemann surface, possibly with boundary, and $\mathbb{M}$ is a finite subset of $\Sigma$.  Elements of $\mathbb{M}$ are called \emph{marked points}; when these lie in the interior of $\Sigma$, they are called \emph{punctures}.  We require that every connected component of the boundary of $\Sigma$ contains at least one marked point.  We also exclude the following cases:
\begin{itemize}
    \item a sphere without boundary and at most three punctures;
    \item a once-punctured monogon;
    \item an unpunctured disc with at most three marked points.
\end{itemize}

We use the same definitions of arcs, tagged arcs, ideal triangulations and tagged triangulations as in \cite{FST}.  For any tagged triangulation $\tau$, we denote by $\tau^\circ$ the corresponding ideal triangulation.

\begin{definition}
Let $\surf$ be a surface with marked points.  For any natural number $v$, a \emph{valency $\geq v$-triangulation} is an ideal triangulation such that all punctures have valency at least $v$.  A  \emph{valency $\geq v$-tagged triangulation} is a tagged triangulation $\tau$ such that $\tau^\circ$ is a \emph{valency $\geq v$-triangulation}.        
\end{definition}


Let $\tau$ be a valency $\geq 2$-triangulation.  Let $Q(\tau)$ and $S(\tau)$ be the quiver and the potential associated to the surface (cf. \cite[Definitions 8 and 23]{LF1}).  We respectively denote by $\Qtau_0$ and $\Qtau_1$ the vertex set and the arrow set of the quiver $\Qtau$. Let $\Qtau_2=\Qtau_2^+\cup\Qtau_2^-$, where $\Qtau_2^+$  is the set of $3$-cycles that appear in $\Stau$ and come from internal triangles of $\tau$, and $\Qtau^-$ is the set of cycles that appear in $\Stau$ and surround the punctures of $\surf$.

If we let $\unredQtau$ and $\unredStau$ be the unreduced quiver and the unreduced potential associated to $\surf$ (cf. \cite[Definitions 8 and 23]{LF1}), then we define in a similar fashion the sets $\unredQtau_0$, $\unredQtau_1$ and $\unredQtau_2^\pm$.  Note that we always have $\Qtau_0=\unredQtau_0$; however, if some puncture has valency 2 with respect to $\tau$, then $\Qtau_1\varsubsetneq \unredQtau_1$, $\Qtau_2^+\varsubsetneq\unredQtau_2^+$ and $\Qtau_2^-\nsubseteq\unredQtau_2^-$.

For each integer $n$, let $\Ctau{n}$ and  $\unredCtau{n}$ be the abelian groups defined by
$$
\Ctau{n} =\begin{cases}
\text{The free abelian group with basis $\Qtau_n$}  & \text{if $n\in\{0,1,2\}$};\\
0 & \text{if $n\notin\{0,1,2\}$}.
\end{cases}
$$
and the same for $\unredCtau{n}$, replacing $\Qtau_n$ by $\unredQtau_n$. Notice that $\Ctau{0}=\unredCtau{0}$ since $\Qtau_0=\unredQtau_0$.

For each integer $n$, let $\partial_n:\Ctau{n}\rightarrow \Ctau{n-1}$  be the group homomorphism defined by
\begin{eqnarray}
\nonumber
\partial_0   &=& 0;\\
\nonumber
\partial_1(a)  &=& i-j \ \ \ \ \  \ \ \ \ \ \ \ \ \ \ \ \ \ \ \ \ \  \text{if $a\in\Qtau_1$, $a:j\rightarrow i$;}\\
\nonumber
\partial_2(\xi)  &=& a_1+a_2+\ldots+a_\ell \ \ \ \ \  \text{if $\xi=a_1a_2\ldots a_\ell\in\Qtau_2$;}\\
\nonumber
\partial_n  &=&0 \ \ \ \ \ \ \ \  \ \ \ \ \ \ \ \ \ \ \ \ \ \ \ \ \ \ \  \text{if $n\notin\{0,1,2\}$.}
\end{eqnarray}
Define the morphism $\widehat{\partial}_{n}:\unredCtau{n}\rightarrow\unredCtau{n-1}$ in a similar way.

It is straightforward to verify that $\partial_{n-1}\circ \partial_n=0$  and $\widehat{\partial}_{n-1}\circ\widehat{\partial}_n=0$ for every $n$. This means that $(\Ctau{\bullet},\partial_\bullet)=((\Ctau{n})_{n\in\ZZ},(\partial_{n})_{n\in\ZZ})$ and $(\unredCtau{\bullet},\widehat{\partial}_\bullet)=((\unredCtau{n})_{n\in\ZZ},(\widehat{\partial}_{n})_{n\in\ZZ})$ are chain complexes of abelian groups. Notice that $\Ctau{\bullet}$ is not a subcomplex of $\unredCtau{\bullet}$ if some puncture has valency 2 with respect to $\tau$.

For $n\in\ZZ$, let $\varphi_n:\unredCtau{n}\rightarrow\Ctau{n}$ be the group homomorphism defined by
\begin{eqnarray}
\nonumber \varphi_0 & = & \myid_{\Ctau{0}}\\
\label{eq:hatCtau_1->Ctau_1}\varphi_1(\alpha) & = & \begin{cases}
\alpha & \text{if the marked point associated to $\alpha$ is not a puncture}\\
 & \text{of valency 2;}\\
\varepsilon_\alpha+\eta_\alpha & \text{otherwise, where we use the notation from Figure \ref{Fig:valency2_arrow_notation}.}
\end{cases}\\
\nonumber
\varphi_2(\xi)&=&\begin{cases}
\xi & \text{if $\xi$ is already a term of $\Stau$};\\
S^p(\tau) & \text{if $\xi=\widehat{S}^\triangle(\tau)$ for a triangle $\triangle$ containing}\\
 & \text{a puncture of valency 2, being $p$ such puncture};\\
S^p(\tau)  & \text{if $\xi=\widehat{S}^p(\tau)$ for a puncture $p$ of valency 2}.
\end{cases}\\
\nonumber
\varphi_n & = & 0 \ \ \ \text{if $n\notin\{0,1,2\}$.}
\end{eqnarray}
\begin{figure}[!h]
                \caption{$\varphi_1(\alpha)=\varepsilon_\alpha+\eta_\alpha$ if the marked point associated to $\alpha$ is a puncture of valency 2}\label{Fig:valency2_arrow_notation}
                \centering
                \includegraphics[scale=.5]{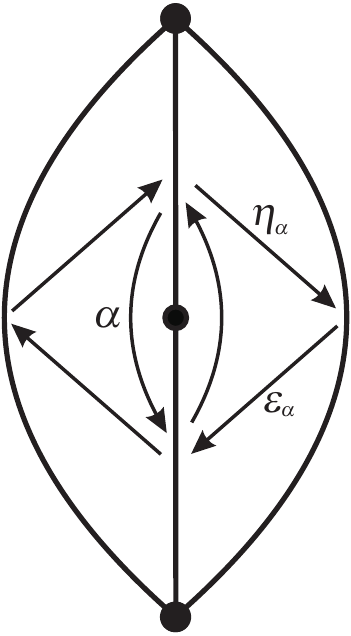}
        \end{figure}

The proof of following lemma is an easy exercise.

\begin{lemma}\label{lemma:phi-unredCtau->Ctau-is-homotopy-equivalence} The collection $\varphi_\bullet=(\varphi_n)_{n\in\ZZ}$ is a homotopy equivalence $\unredCtau{\bullet}\rightarrow\Ctau{\bullet}$.
\end{lemma}

\subsection{The CW-complex $X_{\triangtau}$}

Our aim in this subsection is to realize $(C_\bullet(\triangtau),\partial_\bullet)$ as the cellular chain complex of certain CW-complex whose geometric realization is homeomorphic to a strong deformation retract of $\surface$. 

Suppose $\triang$ is a triangle of $\triangtau$ that contains at least one arrow of the quiver $\Qtau$. Then $\triang$ contains either 1 or 3 arrows of $\Qtau$. We associate to $\triang$ an unpunctured polygon $\surftriang$, together with a labeling of the sides of $\surftriang$, as follows. If $\triang$ contains exactly one arrow $a\in\Qtau_1$, then two of the sides of $\triangle$ are arcs in $\tau$ and the remaining side is a boundary segment. We set $\surftriang$ to be an unpunctured square. As for the side labeling, we pick a side of $\surftriang$ and label it with the arrow $a$, then label the adjacent sides with the arcs $\tail(a)$ and $\head(a)$ in such a way that the label $t(a)$ precedes the label $a$ and the label $a$ precedes the label $h(a)$ in the clockwise direction defined by the orientation of $\surftriang$; the remaining side of $\surftriang$ is then labeled with the boundary segment of $\surface$ that is contained in $\triang$. See Figure \ref{Fig:CW-complex-Xtau}.

If $\triang$ contains exactly three arrows $a,b,c\in\Qtau_1$, then all three sides of $\triangle$ are arcs in $\tau$. We set $\surftriang$ to be an unpunctured hexagon. As for the side labeling, suppose that $\tail(a)=\head(b)$ and $\tail(b)=\head(c)$, then pick a side of $\surftriang$ and label it with the arrow $a$, then label the remaining sides with the arcs and arrows $\head(a)$, $\tail(a)$, $b$, $\tail(b)$ and $c$, in such a way that, according to the orientation of $\surftriang$, the clockwise appearance of the six labels is $(\tail(a),a,\head(a),c,\tail(b),b)$ (up to cyclic permutation). See Figure \ref{Fig:CW-complex-Xtau}.

Now, suppose $p$ is a puncture of $\surf$. We associate to it an unpunctured polygon $\surfpunct{p}$, together with a labeling of the sides of $\surfpunct{p}$, as follows. Let $\xi=a_1\ldots a_{\ell_p}$ be the summand of $\Stau$ that runs around $p$. Set $\surfpunct{p}$ to be an unpunctured $p$-gon. As for the side labeling, the labels are precisely the arrows $a_1,\ldots,a_{\ell_p}$, and they are placed in such a way that, according to the orientation of $\surftriang$, the counterclockwise appearance of the $\ell_p$ labels is $(a_1,a_2,\ldots,a_{\ell_p})$ (up to cyclic permutation). See Figure \ref{Fig:CW-complex-Xtau}.

The set $\mathcal{S}=\{\surftriang\suchthat\triang$ is a triangle of $\tau$ containing at least one arrow of $\Qtau \}\cup \{\surfpunct{p}\suchthat p\in\punct\}$, together with the side labelings of its elements, yields a 2-dimensional CW-complex in an obvious way. 

\begin{definition}
The CW-complex $X_{\tau}$ is the 2-dimensional CW-complex defined above.
\end{definition}

Namely, there is a 2-cell for each element of $\mathcal{S}$, a 1-cell for each element of $\Qtau_0\cup\Qtau_1\cup\{s\suchthat s$ is a boundary segment of $\surf$ and a side of a triangle that contains exactly one arrow of $\Qtau\}$, and two 0-cells for each element of $\Qtau_0$. The attaching maps are defined by identifying pairs of sides of elements of $\mathcal{S}$ that are labeled by the same element of $\Qtau_0\cup\Qtau_1$.

Let us be more precise (and careful) about the way the identification of pairs of sides with the same label is done. Orient each side of each element of $\mathcal{S}$ in a clockwise manner (according to the orientation of the corresponding element of $\mathcal{S}$). If $i$ and $j$ are sides of elements of $\mathcal{S}$ that have the same label, then we glue $i$ and $j$ in such a way that traversing $i$ according to the orientation we have fixed for $i$ corresponds to traversing $j$ according to the orientation which is opposite to the one we have fixed for $j$.

\begin{figure}\caption{The CW complex $X_\tau$}\label{Fig:CW-complex-Xtau}
\[\scalebox{1}{
\begin{tikzpicture}[>=stealth,scale=0.6]

\fill[gray!50] (-1,-1)--(-1,9)--(9,9)--(9,-1)--cycle;
\draw[thick, fill=white] (0,0)--(0,8)--(8,8)--(8,0)--cycle;

\draw[thin, fill=white] (3,3)--(3.33,4)--(4.8,4)--(5,3)--(4,2)--cycle;
\draw[thin, fill=white] (1.33,5)--(1.5,6.5)--(3,6.33)--(2.67,5)--cycle;

\draw[thin, fill=blue!50] (1,2.25)--(1,3)--(0.67,4)--(0.5,7.5)--(6,7.33)--(7.2,7)--(7,3)--(7,2.25)--(5,0.75)--(4,1)--(3,0.75)--cycle;

\draw[thick, fill=blue!70] (1.3,5)--(2.7,5)--(3,6.3)--(1.5,6.5)--(1.3,5);
\draw[thick, fill=blue!70] (4,2)--(5,3)--(5,4.2)--(3.3,4)--(3,3)--(4,2);

\node (A) at (0,0) {$\bullet$};
\node (B) at (4,0) {$\bullet$};
\node (C) at (8,0) {$\bullet$};
\node (D) at (0,3) {$\bullet$};
\node (E) at (4,3) {$\bullet$};
\node (F) at (8,3) {$\bullet$};
\node (G) at (2,6) {$\bullet$};
\node (H) at (0,8) {$\bullet$};
\node (I) at (8,8) {$\bullet$};

\draw[thick] (D)--(B)--(F)--(E)--(I)--(G)--(E)--(D)--(G)--(H);
\draw[thick] (B)--(E);

%
%
%
%
%
%
%
%

\end{tikzpicture}}
\]
\end{figure}
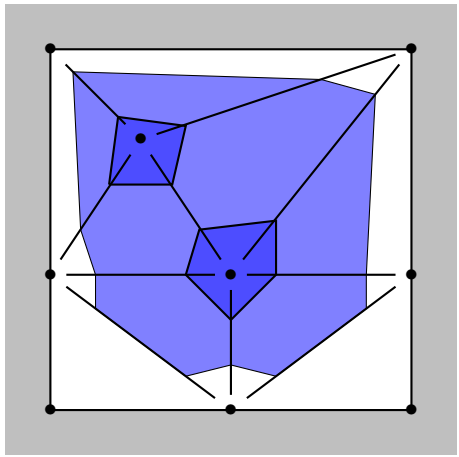

We will write $C_\bullet(X_\tau)$ to denote the cellular chain complex of the CW-complex $(X_\tau,(X_\tau^{(\ell)})_{\ell\geq 0})$ over $\ZZ$. We shall calculate $C_{\bullet}(X_\tau)$ explicitly. Our calculation will take a particularly simple form due to the fact that $(X_\tau,(X_\tau^{(\ell)})_{\ell\geq 0})$ is a regular CW-complexes. Our reference for the calculation we are about to make is \cite{Cooke-Finney} (more specifically, Definition II.1.9 and page 35 in loc. cit.).

For each non-negative integer $n$, $C_{n}(X_\tau)$ is the free abelian group with basis $X_\tau^{(n)}$.

Each of the $1$-cells in $X_\tau^{(1)}$ has been assigned an orientation. This assignment of orientations provides each $1$-cell $e^1$ with a well-defined initial $0$-cell $\initial(e^1)$ and a well-defined terminal $0$-cell $\terminal(e^1)$. Furthermore, given a $2$-cell $e^2$ and a $1$-cell $e_1$, we  define the integer
$$
[e^2:e^1]=\begin{cases}
1 & \text{if $e^1\subseteq \overline{e^2}$ and the orientation assigned to $e^1$ coincides with}\\
& \text{the clockwise orientation around $e^2$;}\\
-1 & \text{if $e^1\subseteq \overline{e^2}$ and the orientation assigned to $e^1$ coincides with}\\
& \text{the counterclockwise orientation around $e^2$;}\\
0 & \text{if $e_1$ is not contained in $\overline{e_2}$.}
\end{cases}
$$
The boundary maps of the cellular chain complex $C_{\bullet}(X_\tau)$ is then given by the rules
\begin{eqnarray*}
\partial_0 & = & 0\\
\partial_1(e_1) &=& \terminal(e^1)-\initial(e^1)\\
\partial_2(e_2) &=& \sum\limits_{e^1\in X_\tau^{(1)}}[e^2:e^1]e^1.
\end{eqnarray*}

Let $\psi_n:C_n(X_\tau)\rightarrow C_n(\tau)$, $n\in\{0,1,2\}$, be the group homomorphisms defined on the corresponding bases by
\begin{eqnarray*}
\psi_0(e^0)&=&\text{the unique element in $Q(\tau)_0=\tau$ containing $e^0$;}\\
\psi_1(e^1)&=&\begin{cases}
a & \text{if $e^1$ is parallel to the arrow $a\in Q(\tau)_1$ and}\\
& \text{it is oriented in the same direction as $a$;}\\
-a & \text{if $e^1$ is parallel to the arrow $a\in Q(\tau)_1$ and}\\
&  \text{it is oriented in the direction opposite to that of $a$;}\\
0 & \text{if $e^1$ is contained in an arc of $\tau$;}
\end{cases}\\
\psi_2(e^2)&=& \begin{cases}
\text{the term of $S(\tau)$ arising from $e^2$} & \text{if $e^2=e^2_\triangle$ for some interior triangle $\triangle$ of $\tau$}\\
& \text{or $e^2=e^2_p$ for some puncture $p$;}\\
0 & \text{if $e^2=e^2_\triangle$ for some non-interior triangle}\\
&\text{$\triangle$ of $\tau$}.
\end{cases}
\end{eqnarray*}
Setting $\psi_n=0$ for $n>2$, a routine check shows that $\psi_\bullet=(\psi_n)_{n\in\mathbb{Z}_\geq 0}$ is a morphism of chain complexes $C_\bullet(X_\tau)\rightarrow C_\bullet(\tau)$.

\begin{proposition}\label{prop:C(Xtau)->C(tau)-isomorphism-in-homology} The maps induced by $\psi_\bullet$ on homology are isomorphisms. That is, for all $n\in\mathbb{Z}$, the group homomorphism $\overline{\psi}_n:H_n(C_\bullet(X_\tau))\rightarrow H_n(C_\bullet(\tau))$ given by $\overline{\psi}_n[\zeta]=[\psi_n(\zeta)]$ is an isomorphism.
\end{proposition}

\begin{proof}
Esentially, the proof of \cite[Lemma 2.3]{Amiot-Grimeland} given by Amiot-Grimeland for unpunctured surfaces works here.
\end{proof}

\begin{corollary} Let $\tau$ be a valency $\geq 2$-triangulation of $\surf$. There is a canonical isomorphism of abelian groups between $H_1(C_\bullet(\triangtau),\partial_\bullet)$ and the singular homology group $H_1(\Sigma,\mathbb Z) $. Hence each choice of $\mathbb{Z}$-basis of $H_1(\Sigma,\mathbb Z)$ induces an isomorphism $H_1(C_\bullet(\triangtau),\partial_\bullet)\cong \mathbb Z^{2g+b-1}$, where $g$ is the genus of $\Sigma$ and $b$ is its number of boundary components.
\end{corollary}

\subsection{Evaluating gradings at curves}


\begin{definition} Let $\tau$ be a valency $\geq 3$-triangulation. A closed curve $\gamma:[0,1]\rightarrow\Sigma\setminus(\mathbb{M}\cup\partial\Sigma)$ is \emph{$\tau$-admissible} if it satisfies the following contitions:
\begin{enumerate}
\item the intersection of $\gamma$ with each arc of $\tau$ is a finite set;
\item all the intersection points of $\gamma$ with any arc of $\tau$ are transversal crossings;
\item $\gamma$ does not cross any arc of $\tau$ twice in succession.
\end{enumerate}
\end{definition}

The proof of \cite[Lemma 2.3]{Amiot-genus1} can be easily adapted to show the following:

\begin{lemma}\label{lemma:nice-representative-curves} Let $\tau$ be a valency $\geq 3$-triangulation of $\surf$. Every closed curve $x$ on $\Sigma\setminus\mathbb{M}$ is freely homotopic to a $\tau$-admissible closed curve $\gamma_x$.
\end{lemma}

For any $\tau$-admissible closed curve $\gamma$, define $\widetilde{\gamma}^\tau\in \unredCtau{1}=\mathbb{Z}\unredQtau_1$ to be the signed algebraic sum of the arrows traversed by $\gamma$. The proof of \cite[Lemma 2.5]{Amiot-genus1} can be adapted to show that the rule $x\mapsto\widetilde{\gamma_x}^\tau$ induces a well-defined map $\pi_1(\Sigma\setminus\mathbb{M})\rightarrow\mathbb{Z}\unredQtau_1$. 
We define
$$
\overline{(-)}^\tau=\varphi_1\circ\widetilde{(-)}^\tau:\pi_1(\Sigma\setminus\mathbb M)\to  \mathbb{Z}\Qtau_1,
$$
where $\varphi_1:\mathbb{Z}\unredQtau_1\rightarrow\mathbb{Z}\Qtau_1$ is the group homomorphism defined right before Lemma \ref{lemma:phi-unredCtau->Ctau-is-homotopy-equivalence}.

\begin{remark} The map $\pi_1(\Sigma\setminus\mathbb{M})\rightarrow\mathbb{Z}Q(\triangtau)_1$ just described may not be a group homomorphism. It does, however, factor through the set $\pi_1^{\textrm{free}}(\Sigma\setminus\mathbb{M})$ of free homotopy classes of closed curves.
\end{remark}

Let $\tau$ be a valency $\geq 2$-triangulation of $(\Sigma,\marked)$. We denote by $(C^\bullet(\triangtau),\delta^{\bullet})=((C^n(\triangtau))_{n\in\ZZ},(\delta^{n})_{n\in\ZZ})$ the cochain complex which is dual to $(C_\bullet(\triangtau),\partial_\bullet)$ over $\ZZ$. Thus, for every integer $n\in\ZZ$, the group $C^n(\triangtau)$ and the map $\delta^{n}:C^n(\triangtau)\rightarrow C^{n+1}(\triangtau)$ are defined by
$$
C^n(\triangtau)=\Hom_{\ZZ}(C_n(\triangtau),\ZZ) \ \ \ \text{and} \ \ \ \delta^n(f)=f\circ\partial_{n+1}.
$$

According to Definition \ref{def:equivalence-of-quiver-gradings}, a grading on $\Qtau$ is any function $d:\Qtau_1\rightarrow\ZZ$. Since $\Qtau_1$ is a $\ZZ$-basis of $C_1(\triangtau)$, any grading $d:\Qtau_1\rightarrow\ZZ$ induces a group homomorphism $C_1(\triangtau)\rightarrow\ZZ$, that is, an element of $C^1(\triangtau)$. In a clear but harmless abuse, we will denote this element of $C^1(\triangtau)$ by $d$ as well.

%
%
%

\begin{definition} Let $\tau$ be a tagged (resp. valency $\geq 2$-) triangulation and $d$ be a grading on $\Qtau$. We will say that $d$ is a \emph{degree-1 map} for $\tau$, or that the pair $(\tau,d)$ is a \emph{graded tagged triangulation} (resp. \emph{graded valency $\geq 2$-triangulation}), if the potential $S(\tau)$ defined in \cite[Definitions 3.1 and 3.2]{LF4} is homogeneous of degree 1 with respect to $d$.
\end{definition}

Notice that if $\tau$ is a valency $\geq 2$-triangulation, then a grading $d$ on $\Qtau$ is a degree-1 map if and only if $(d\circ\partial_2)(\xi)=1$ for every $\xi\in\Qtau_2$ (recall that $d\circ\partial_2=\delta^1(d):C_2(\triangtau)\rightarrow\ZZ$).

So, if $\tau$ is a valency $\geq 2$-triangulation, then the difference of two degree-1 maps vanishes on the image of $\partial_2$, which is the same as saying that the difference of any two degree-1 maps is a 1-cocycle in $C^\bullet(\tau)$. 

Recall that we have group isomorphisms $H^1(C^\bullet(\tau))\rightarrow \Hom_{\mathbb{Z}}(H_1(C_\bullet(\tau)),\mathbb{Z})$, $H_1(\Sigma)\rightarrow H_1(C_\bullet(\tau))$ and $\Hom_{\mathbb{Z}}(H_1(\Sigma),\mathbb{Z})\rightarrow H^1(\Sigma)$. If $d_1$ and $d_2$ are degree-1 maps, then we have a well-defined cohomology class $[d_1-d_2]^\tau\in H^1(C^\bullet(\tau))$, and we can think of this class as a group homomorphism $H_1(C_\bullet(\tau))\rightarrow\mathbb{Z}$, hence also as a group homomorphism $H_1(\Sigma)\rightarrow \mathbb{Z}$, which in turn corresponds to a cohomology class in $H^1(\Sigma)$. We will denote this cohomology class by $[d_1-d_2]_{H^1(\Sigma)}$.

\begin{lemma}\label{lemma:equivalence-of-gradings-characterization}  Let $\tau$ be a valency $\geq 2$-triangulation of $\surf$.
For two degree-1 maps $d_1$ and $d_2$ on $Q(\tau)$, the following are equivalent:
\begin{enumerate}
\item $d_1$ and $d_2$ are equivalent gradings on $Q(\tau)$;
\item $[d_1-d_2]^\tau=0$ as element in $H^1(C^\bullet(\tau))$;
\item $d_1(\overline{\gamma}^\tau)=d_2(\overline{\gamma}^\tau)$ for every $\tau$-admissible closed curve $\gamma$ on $\Sigma\setminus\mathbb M$;
\item there exists a set $\Gamma$ of $\tau$-admissible closed curves on $\Sigma\setminus\mathbb{M}$ such that:
\begin{enumerate}\item $\{[\overline{\gamma}^\tau]\suchthat \gamma\in\Gamma\}$ is a $\mathbb{Z}$-basis of $H_1(C_\bullet(\tau))$;
 \item $d_1(\overline{\gamma}^\tau)=d_2(\overline{\gamma}^\tau)$ for all $\gamma\in\Gamma$;
 \end{enumerate}
\item $[d_1-d_2]_{H^1(\Sigma)}=0$ as element in $H^1(\Sigma)$.
\end{enumerate}
\end{lemma}

\subsection{Graded flip}\label{subsec:graded-flip}


Let $(\triangtau,d)$ be a graded tagged triangulation. By \cite[Corollary 9.1]{LF4}, $(Q(\tau),S(\tau))$ is non-degenerate. Hence there is a well-defined grading $\mu_{k_\ell}^L\ldots\mu_{k_1}^L(d)$ on $\mu_{k_\ell}^L\ldots\mu_{k_1}^L(Q(\tau))=\mu_{k_\ell}\ldots\mu_{k_1}(Q(\tau))$ for any finite sequence $(k_1,\ldots,k_\ell)$ by Lemma \ref{prop:mutation-of-graded-QPs}.

\begin{definition}\label{defi:graded-flip} Let $(\triangtau,d)$ be a graded tagged triangulation. For $k\in\tau=\Qtau_0$ we call the pair $\flip_k(\tau,d):=(\flip_k(\tau),\mu_k^L(d))$ the (left) \emph{graded flip} of $(\tau,d)$ with respect to $k$, where $\flip_k(\tau)$ is the tagged triangulation obtained from $\tau$ by flipping the arc $k$.
\end{definition}

The next lemma is a version of \cite[Lemma 2.14]{Amiot-Grimeland} for punctured surfaces.

\begin{lemma}\label{lemma:graded-flip-preserves-evaluation}
Let $(\triangtau,d)$ be a graded valency $\geq 2$-triangulation, and let $i\in\tau$ be an arc such that $\flip_i(\tau)$ is a valency $\geq 2$-triangulation. For any closed curve $\gamma$ on $\Sigma\setminus \mathbb M$, we have
$$
d(\overline{\gamma}^\tau)=\mu_i^L(d)(\overline{\gamma}^{\flip_i(\tau)}).
$$
\end{lemma}

Let $(\triangtau,d)$ be a graded valency $\geq 2$-triangulation, and let $\triangtau'$ be any valency $\geq 2$-triangulation. According to \cite[Equation (6.4)]{LF2}, $\triangtau$ and $\triangtau'$ can be connected by a sequence of flips involving only valency $\geq 2$-triangulations. Using any such sequence, we can use Definition \ref{defi:graded-flip} to define a grading on $Q(\triangtau')$. This grading is a degree-1 map (this follows, for instance, from Theorem \ref{thm:flip=>graded-right-equivalence}, which we prove later). Of course, different sequences of flips may yield different degree-1 maps on $Q(\triangtau')$, but these have to be equivalent gradings by a combination of Lemmas \ref{lemma:equivalence-of-gradings-characterization} and \ref{lemma:graded-flip-preserves-evaluation}. If no confusion shall arise, we will write $\{d\}_{\triangtau'}$ for the degree-1 map on $Q(\triangtau')$ obtained by applying any sequence of flips.

\begin{corollary}\label{corollary3.13}
Let $(\tau,d)$ and $(\tau',d')$ be two graded valency $\geq 2$-triangulations. Then $\{d\}_{\triangtau'}-d'$ induces a function $(\{d\}_{\triangtau'}-d')\circ\overline{(-)}^{\tau'}:\pi_1(\Sigma\setminus \mathbb M)\to \mathbb Z$ and the diagram
$$
\xymatrix{
\pi_1(\Sigma\setminus \mathbb M) \ar[r] \ar[dr]_{(\{d\}_{\triangtau'}-d')\circ\overline{(-)}^{\tau'}} & H_1(\Sigma\setminus\mathbb{M}) \ar[r] & H_1(\Sigma) \ar[dl]^{[\{d\}_{\triangtau'}-d']_{H^1(\Sigma)}}\\
& \ZZ & 
}
$$
commutes.
\end{corollary}

\subsection{Homeomorphisms and graded equivalence}

Let $\Phi:\Sigma\to \Sigma$ be an orientation-preserving homeomorphism such that $\Phi(\mathbb{M}) = \mathbb{M}$.  It sends arcs to arcs and triangulations to triangulations.  Moreover, if $\tau$ is a triangulation and $\tau'$ is its image by $\Phi$, then $\Phi$ induces an isomorphism of quivers $Q(\tau)\to Q(\tau')$, which we still denote by $\Phi$.

\begin{proposition}[Proposition 8.5 of \cite{Bridgeland-Smith}, graded version]\label{prop::BS}
Assume that $(\Sigma, \mathbb{M})$ is not one of the following:
\begin{itemize}
 \item A sphere with at most $5$ marked points;
 \item An unpunctured disc with at most $3$ marked points;
 \item a once-punctured disc with $1$, $2$ or $4$ marked points on the boundary;
 \item a twice-punctured disc with $2$ marked points on the boundary.
\end{itemize}
Let $(\tau, d)$ and $(\tau', d')$ be two graded triangulations of $\surf$.  Then there exists an orientation preserving homeomorphism $\Phi:\Sigma\to \Sigma$ such that $\Phi(\mathbb{M}) = \mathbb{M}$, $\Phi(\tau) = \tau'$ and $d=d'\circ\Phi$ if and only if the graded quivers associated to $(\tau, d)$ and $(\tau', d')$ are isomorphic.
\end{proposition}

\begin{proof}

 We give a graded version of the proof in \cite{Bridgeland-Smith}.  One implication is obvious: if there exists a homeomorphism $\Phi$ as in the statement, then the graded quivers associated to the graded triangulations are isomorphic.

 To prove the converse, we use a graded version of the block decomposition of \cite{FST}.  In \cite[Remark 4.1]{FST}, it is shown that any ideal triangulation (except for the four-punctured sphere) can be obtained by gluing three different types of ``puzzle pieces'' together; it is easily seen that graded triangulations are obtained by gluing graded versions of these puzzle pieces.  Next, it was shown in \cite[Theorem 13.3]{FST} that a quiver is the quiver of some ideal triangulation if and only if the quiver is obtained by gluing five different types of ``blocks'' together; moreover, the blocks in the decomposition correspond to puzzle pieces.  Again, it is easy to see that a graded quiver will be associated to a graded triangulation if and only if it is obtained by gluing graded versions of the five types of blocks.

While the quiver of a triangulation is not enough to recover the triangulation itself, the data of the quiver together with its block decomposition is enough to recover it up to an orientation-preserving homeomorphism; this is explained in the proof of \cite[Theorem 13.3]{FST}.  Thus if the (graded) quiver of a triangulation has a unique block decomposition, then any two triangulations $(\tau, d)$ and $(\tau', d')$ giving rise to this quiver are related by a homeomorphism $\Phi$ as in the statement (that is, $\tau'=\Phi(\tau)$ and $d=d'\circ\Phi$).  An algorithm to determine the block decompositions of a given quiver was developed in \cite{Gu1} and used in \cite{Gu2} to list all quivers with multiple block decompositions.  We see from this list that the surfaces admitting different triangulations with the same quiver are those excluded in the statement of the proposition.
\end{proof}

To any tagged triangulation $\triangtau$ we associate the Jacobian algebra $\mathcal{P}(\Qtau,S(\tau))$. Then any degree-1 map on $\Qtau$ gives rise to a grading on the algebra $\mathcal{P}(Q(\triangtau),S(\triangtau))$.

\begin{proposition}\label{fundamental}
Let $(\tau,d)$ and $(\tau',d')$ be two graded valency $\geq 2$-triangulations of $\surf$. Then the following are equivalent.
\begin{enumerate}
\item The algebras $\mathcal{P}(Q(\tau),S(\tau),d)$ and $\mathcal{P}(Q(\tau'),S(\tau'),d')$ are graded equivalent;
\item There exists an orientation preserving homeomorphism $\Phi:\Sigma\to \Sigma$ with $\Phi(\mathbb M)=\mathbb M$ such that $\Phi(\tau)=\tau'$ and for any simple closed curve $\gamma$ on $\Sigma$, we have $d(\overline{\gamma}^\tau)=d'(\overline{\Phi(\gamma)}^{\tau'})$.
\end{enumerate}
\end{proposition}

\begin{proof}

That (2) implies (1) is immediate.  Assume now that (1) holds.  Then there is an isomorphism of quivers $\phi:Q(\tau)\to Q(\tau')$.  According to Remark \ref{rema::graded-equiv}, $d$ and $d'\circ\phi$ are equivalent gradings.  

Note that $d$ and $d'\circ\phi$ give the same result when evaluated at closed curves.  Note also that the graded quivers of $(\tau,d'\circ\phi)$ and $(\tau',d')$ are isomorphic.  Thus Proposition \ref{prop::BS} implies (2) for all cases except possibly the once-punctured disc with $4$ marked points on the boundary and the twice-punctured disc with $2$ marked points on the boundary (indeed, valency $\geq 2$-triangulations of surfaces with non-empty boundary do not fall within the other excluded cases of Proposition \ref{prop::BS}).

For the once-punctured disc with $4$ marked points on the boundary, there are only finitely many triangulations, exactly five of which (up to orientation-preserving homeomorphisms) are valency $\geq 2$-triangulations.  It is easy to check directly on these that (2) holds

For the twice-punctured disc with $2$ marked points on the boundary, there are (up to orientation-preserving homeomorphisms) only four valency $\geq 2$-triangulations.



\end{proof}


\section{Graded flip vs graded mutation}\label{section4}


Let $(\tau,d)$ be a graded tagged triangulation and $i$ be any arc in $\tau$. By definition the potential $\Stau$ is homogeneous of degree 1 with respect to $d$. Hence by Proposition \ref{prop:mutation-of-graded-QPs}, the potential $\mu_i^L(S(\triangtau))$ is homogeneous of degree 1 with respect to $\mu_i^L(d)$. Moreover by \cite[Theorem 30]{LF1}, we know that  $(\mu_i(Q(\tau)),\mu_i(\Stau))$ and $(Q(\flip_i(\tau)),S(\flip_i(\tau)))$ are right-equivalent. 
Can the right equivalence be chosen so as to be an isomorphism of graded algebras with respect to $\mu_i^L(d)$ ? In this section we prove the following result that answers positively this question.

\begin{theorem}\label{thm:flip=>graded-right-equivalence} Let $(\tau,d)$ be a graded tagged triangulation and $i$ be a tagged arc in $\tau$. Then there exists a graded right equivalence $$\Theta:\mu_i^L((Q(\tau)),\Stau,d))\underset{\mathbb Z}{\cong}(Q(\flip_i(\triangtau)),S(\flip_i(\triangtau)),\mu_i^L(d)).$$

Consequently, $S(\flip_i(\triangtau))$ is homogeneous of degree 1 with respect to $\mu_i^L(d)$, and  there is an isomorphism of graded algebras $$ \mathcal{P}(\mu_i(Q(\tau)),\mu_i(\Stau),\mu_i^L(d)) \underset{\mathbb Z}{\cong}\mathcal{P}(Q(\flip_i(\triangtau)),S(\flip_i(\triangtau)),\mu_i^L(d)).$$ 
\end{theorem}

The proof of this result, given in Subsection \ref{subsection:proof}, involves checking that the right equivalence between $(\mu_i(Q(\tau)),\mu_i(\Stau))$ and $(Q(\flip_i(\triangtau)),S(\flip_i(\triangtau)))$ constructed in the proof of \cite[Theorem 8.1]{LF4} is graded with respect to $\mu_i^L(d)$. Since this latter proof is quite long, and since it in turn depends heavily on the proof of \cite[Theorem 30]{LF1} and on several technical statements, we devote Subsection \ref{subsection:proof} to describing how the reader can verify this. We first show that we can restrict to the case where $\tau$ is an ideal triangulation. Then we consider two cases depending wether $\flip_i(\tau)$ is ideal or not. This latter case is the most delicate. To handle it, we need intermediate results (Corollaries~\ref{cor:psiW-homogenous} and \ref{cor:psi-conjugacy}) that follow from Proposition \ref{prop:pop-is-quiver-grading-equivalence}.

\subsection{The automorphism $\psi_{i,j}$}

\begin{proposition}\label{prop:pop-is-quiver-grading-equivalence}
Let $\nu$ be an ideal triangulation. Suppose that $i\in\nu$ is the folded side of a self-folded triangle of $\nu$, and that $j\in\nu$ is the arc that, together with $i$, forms such triangle.  Let $\psi_{i,j}$ be the $\CC$-algebra automorphism of $\RA{Q(\nu)}$ induced by the quiver automorphism of $Q(\nu)$ that interchanges $i$ and $j$. For any grading $\widetilde{d}:Q(\nu)\rightarrow\ZZ$ with respect to which $S(\nu)$ is homogeneous of degree 1, define $\widetilde{d}':= \widetilde{d}\circ \psi_{i,j}$. Then $\widetilde{d}$ and $\widetilde{d}'$ are equivalent as quiver gradings (see Definition \ref{def:equivalence-of-quiver-gradings}).
\end{proposition}

\begin{proof} The quiver automorphism of $Q(\nu)$ can be extended to a quiver automorphism of $\widehat{Q}(\nu)$ where $\widehat{Q}(\nu)$ is the unreduced signed adjacency quiver of $\nu$ as in Subsection \ref{subsection3.1}.  We denote again by $\psi_{i,j}$ the corresponding $\CC$-algebra automorphism of $\RA{\widehat{Q}(\nu)}$.

The grading $\widetilde{d}:Q(\nu)\rightarrow\ZZ$ can be extended in a unique way to a grading $\widehat{Q}(\nu)\rightarrow\ZZ$ with respect to which $\widehat{S}(\nu)$ is homogeneous of degree 1. We denote this extension by $\widetilde{d}$ as well. To prove the lemma it is enough to show that the gradings $\widetilde{d}$ and $\widetilde{d}'=\widetilde{d}\circ \psi_{i,j}:\widehat{Q}(\nu)\rightarrow\ZZ$ are equivalent.

The self-folded triangle formed by $i$ and $j$ is contained in a digon. Assume first that both sides of this digon are arcs in $\nu$ (and not segments of the boundary of $\Sigma$). 
Then the full subquiver of $\widehat{Q}(\nu)$ formed by the arcs that are connected to $i$ or $j$ by arrows of $\widehat{Q}(\nu)$ is one of the three depicted in Figure \ref{Fig:sf_triangles}.
\begin{figure}[!h]
                \caption{Possibilities for a self-folded triangle and the digon surrounding it}\label{Fig:sf_triangles}
                \centering
                \includegraphics[scale=.5]{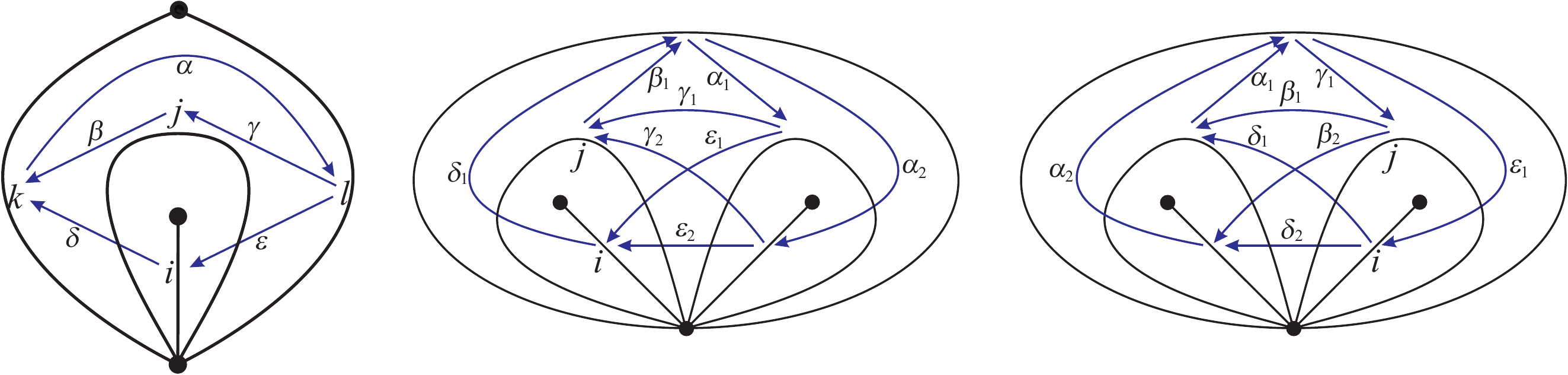}
        \end{figure}
        
 In the first case we have $\widetilde{d}'(\beta)=\widetilde{d}\circ \psi_{i,j}(\beta)=\widetilde{d}(\delta)$.  Similarly we have $\widetilde{d}'(\delta)=\widetilde{d}(\beta)$,  $\widetilde{d}'(\gamma)=\widetilde{d}(\epsilon)$,  $\widetilde{d}'(\epsilon)=\widetilde{d}(\gamma)$ and  $\widetilde{d}'(a)=\widetilde{d}(a)$ for any other arrow. Then define $r:\widehat{Q}(\nu)_0\to \mathbb Z$ to be $$r(i)=\widetilde{d}(\beta)-\widetilde{d}(\delta),\quad r(j)=\widetilde{d}(\delta)-\widetilde{d}(\beta), \textrm{ and } r(m)=0 \textrm{ for any other vertex}.$$ Since $\widehat{S}(\nu)$ is homogenous of degree 1 with respect to $\widetilde{d}$, and since and $\alpha\beta\gamma$ and $\alpha\delta\varepsilon$ are terms that appear in $\widehat{S}(\nu)$ with non-zero coefficient, we deduce that
$$
\widetilde{d}(\beta)+\widetilde{d}(\gamma)=\widetilde{d}(\delta)+\widetilde{d}(\varepsilon).$$ One then easily checks that $$ \widetilde{d}'(a)=\widetilde{d}(a)+r(t(a))-r(h(a))$$ for any arrow $a$ of $Q$ which means that $\widetilde{d}$ and $\widetilde{d}'$ are equivalent gradings.

In the second case, since $\alpha_1\beta_1\gamma_1$, $\alpha_1\delta_1\epsilon_1$, $\alpha_2\delta_1\epsilon_2$ and $\alpha_2\delta_1\epsilon_2$ are terms in the potential, one deduces that $\widetilde{d}(\gamma_2)+\widetilde{d}(\epsilon_1)=\widetilde{d}(\gamma_1)+\widetilde{d}(\epsilon_2).$ Then the map $r:\widehat{Q}(\nu)_0\to \mathbb Z$ defined by $$r(i)=\widetilde{d}(\epsilon_2)-\widetilde{d}(\gamma_2)\quad r(j)=\widetilde{d}(\gamma_2)-\widetilde{d}(\epsilon_2) \textrm{ and } r(k)=0 \textrm{ for } k\neq i,j$$ satisfies $ \widetilde{d}'(a)=\widetilde{d}(a)+r(t(a))-r(h(a))$ for any arrow $a$ of $\widehat{Q}(\nu)$.

The third case and the case where one side of the digon is a segment of the boundary are similar. 
        
\end{proof}

Combining this proposition together with Lemma \ref{lemma:d1_d2-equiv=>muL(d1)_muL(d2)-equiv} and the fact that $\psi_{i,j}=\psi_{i,j}^{-1}$, one has the following. 
\begin{corollary}\label{cor:psiW-homogenous}
Let $\nu$, $i$ and $j$ be as in Proposition \ref{prop:pop-is-quiver-grading-equivalence}, and $\widetilde{d}$ a grading on $Q(\nu)$. Then a potential $W$ on $Q(\nu)$ is  homogeneous of degree 1 with respect to $\widetilde{d}$ if and only if $\psi_{i,j}(W)$ is homogeneous of degree 1 with respect to $\widetilde{d}$.
\end{corollary}

\begin{corollary}\label{cor:psi-conjugacy} 
Let $\nu$, $i$ and $j$ be as in Proposition \ref{prop:pop-is-quiver-grading-equivalence}, $d$ be a grading on $Q(\nu)$ and $\rho:(\mathbb C \langle\langle Q(\nu)\rangle\rangle,d) \to \mathbb (C \langle\langle Q(\nu)\rangle\rangle,d)$ be a graded automorphism that is identity on the vertices. Then the automorphism $\rho':=\psi_{i,j}\circ\rho\circ\psi_{i,j}^{-1}$ is also graded with respect to $d$.
\end{corollary}

\begin{proof}
By Proposition \ref{prop:pop-is-quiver-grading-equivalence}, there exists a map $r:Q(\nu)_0\to \mathbb Z$ such that for any arrow $a$ in $Q(\nu)$ we have $d(\psi_{i,j}(a))=d(a)+r(t(a))-r(h(a)).$
Since $\psi_{i,j}=\psi_{i,j}^{-1}$ this map can be chosen in such a way that $r(\psi_{i,j}(k))=-r(k)$ for each vertex $k$ of $Q(\nu)$ (see for example the choice of $r$ in the proof of Proposition \ref{prop:pop-is-quiver-grading-equivalence}).
 
Then we deduce that for any vertices $k$ and $\ell$, the automorphism $\psi_{i,j}$ sends any homogenous element $w$ which is a linear combination of paths from $k$ and $\ell$ to an homogenous element of degree  $d(\psi_{i,j}(w))=d(w)+r(k)-r(\ell)$ which is a linear combination of paths from $\psi_{i,j}(k)$ to $\psi_{i,j}(\ell)$. 
 
Let $\alpha$ be a path from $k$ to $\ell$ in $Q(\nu)$. Then $\psi_{i,j}(\alpha)$ is a path from $\psi_{i,j}(k)$ to $\psi_{i,j}(\ell)$ in $Q(\nu)$, hence is homogenous with respect to $d$ and has degree $d(\psi_{i,j}(\alpha))=d(\alpha)+r(k)-r(\ell)$. Now since $\rho$ is a graded automorphism which is identity on the vertices, $\rho(\psi_{i,j}(\alpha))$ is a linear combination of paths from $\psi_{i,j}(k)$ to $\psi_{i,j}(\ell)$ and is homogenous of degree $d(\alpha)+r(k)-r(\ell)$.  Hence $\psi_{i,j}\circ\rho\circ\psi_{i,j}(\alpha)$ is a linear combination of paths from $k$ to $\ell$ and is homogenous of degree:
$$d(\psi_{i,j}\circ\rho\circ\psi_{i,j}(\alpha))=(d(\alpha)+r(k)-r(\ell))+r(\psi_{i,j}(k))-r(\psi_{i,j}(\ell))=d(\alpha).$$

This finishes the proof.  

\end{proof}

\subsection{Proof of Theorem \ref{thm:flip=>graded-right-equivalence}}\label{subsection:proof}

Recall that to a general tagged triangulation $\tau$, Fomin-Shapiro-Thurston associate an ideal triangulation $\tau^\circ$ (cf. the second paragraph of \cite[Definition 9.2]{FST}). The quiver $Q(\tau)$ is then defined to be the quiver of $\tau^\circ$, but with $\tau$ as vertex set rather than $\tau^\circ$ (cf. \cite[Definition 9.6]{FST}).  It can be easily verified (with the aid of \cite[Examples 27, 28 and subsequent paragraphs]{LF1} if necessary) that $\Stau$ can be obtained from $S(\tau^\circ)$ by changing the signs of some of the terms that arise from punctures. Hence the potential $S(\tau^\circ)$ is homogenous of degree $1$ with respect to $d$ and the right equivalence $(Q(\tau),S(\tau))\cong (Q(\tau^\circ),S(\tau^\circ))$ (see \cite[Proposition 10.2]{LF4}) obviously preserves the grading $d$. For the same reasons we also have a graded right equivalence 

$$(Q(\flip_i(\tau)),S(\flip_i(\tau)),\mu_i^L(d))\underset{\mathbb Z}{\cong}(Q(\flip_{i^\circ}(\tau^\circ)),S(\flip_{i^\circ}(\tau^\circ)),\mu_i^L(d)),$$ where $i^\circ$ is the arc corresponding to $i$ in $\tau^\circ$.  Therefore, we can assume that $\tau$ is an ideal triangulation.

\medskip

\noindent
\emph{Case 1: $i$ is not the folded side of a self-folded triangle.}

This case corresponds to the case where $\flip_i(\tau)$ is also an ideal triangulation. Then there exists a right equivalence between $(\mu_i(Q(\tau)),\mu_i(\Stau))$ and $(Q(\flip_i(\triangtau)),S(\flip_i(\triangtau)))$ which is explicitely described in the proof of  \cite[Theorem 30]{LF1}, depending on the local configuration around $i$. It is not difficult to check that in each case the right equivalence is compatible with any grading on  $\mu_i(Q(\tau))$ such that $\mu_i(\Stau))$ is homogenous of degree $1$.

\medskip

\noindent
\emph{Case 2: $i$ is the folded side of a self-folded triangle.}

Let $j\in\tau$ be the loop that encloses $i$ and that, together with the latter, forms the alluded self-folded triangle. 
Then there exists a right equivalence $$\Theta:(Q(\tau),S(\tau))\rightarrow\mu_i^{-1}(Q(\flip_i(\tau)),S(\flip_i(\tau)))$$ by  \cite[Theorem 8.1]{LF4}. The proof of this result uses auxillary potentials $W(\tau',i,j)$ associated to any ideal triangulation that contains $i$ and $j$.  
The right equivalence $\Theta$ is defined to be the following composition: 
\[\begin{array}{ll}
(Q(\tau),S(\tau),d)  & =(Q(\flip_s(\sigma)),S(\flip_s(\sigma)),d) \\
& \overset{\Theta_1}{\longrightarrow} \mu_s^L(Q(\sigma),S(\sigma),(\mu_s^{L})^{-1}(d)) \\ 
 & \overset{\Theta_2}{\longrightarrow} \mu_s^L(Q(\sigma),W(\sigma,i,j),(\mu_s^{L})^{-1}(d))  \\
  &  \overset{\Theta_3}{\longrightarrow} (Q(\flip_s(\sigma)),W(\flip_s(\sigma),i,j),d) = (Q(\tau),W(\tau,i,j),d) \\
   
    & \overset{\Theta_4}{\longrightarrow} (\mu_i^L)^{-1}(Q(\flip_i(\tau)),S(\flip_i(\tau)), \mu_i^L(d))
   
\end{array}\]
where $\sigma$ is a particular ideal triangulation containing $i$ and $j$ as in \cite[Proposition 6.4]{LF4} and $s=(i_1,\ldots,i_\ell)$ is a sequence of flips relating $\sigma$ to $\tau=\flip_s(\sigma)$ that does not contain $i$ and $j$ and such that any intermediate triangulation is ideal. 

Let us now explain why all these right equivalences are graded. 

The right equivalence $\Theta_1$ is graded by Case 1, since all intermediate triangulations between $\tau$ and $\sigma$ are ideal. 

By definition, and since $\sigma$ is a triangulation as in \cite[Proposition 6.4]{LF4}, $W(\sigma,i,j)$ is the image under $\psi_{i,j}$ of $S(\sigma)$ which is homogeneous of degree 1 with respect to $(\mu_s^L)^{-1}(d)$. Consequently, $W(\sigma,i,j)$ is homogeneous of degree 1 with respect to $(\mu_s^L)^{-1}(d)$ by Corollary \ref{cor:psiW-homogenous}. The right equivalence $\Theta_2$ is defined in the proof of \cite[Proposition 6.4]{LF4} as a limit $\lim_{n\to\infty}\varphi_n\ldots\varphi_1$ of right equivalences. One checks by induction that each $\varphi_n$ is graded. 

The right equivalence $\Theta_3$ is the composition $\psi_{i,j}\circ\Phi\circ\psi_{i,j}^{-1}$  where $\Phi$ is the right equivalence $\mu_s^L(Q(\sigma),S(\sigma),(\mu_s^L)^{-1}(d))\longrightarrow (Q(\flip_s(\sigma)),S(\flip_s\sigma),d)$ which is graded by Case 1. Therefore by Corollary \ref{cor:psi-conjugacy} the right equivalence $\Theta_3$ is graded.

Finally the right equivalence $\Theta_4$ follows from \cite[Proposition 7.1]{LF4}. The proof is divided in 4 cases, in each of which an explicit right equivalence is exhibited. In all 4 cases, the exhibited right-equivalence sends each arrow to a scalar multiple of itself, and thus is graded.

Consequently, the right equivalence $\Theta$ is graded and Theorem \ref{thm:flip=>graded-right-equivalence} follows.

\qed

We end this section by the following result which is a consequence of Proposition \ref{prop:pop-is-quiver-grading-equivalence} and which will be used in the proof of Theorem \ref{theorem cluster cat}.

\begin{lemma}\label{lemma:tau-and-tau-circ-equivalent-grading} Let $(\tau,d)$ be a graded valency $\geq 2$-tagged triangulation. Then there exists a sequence of flips $s$ such that $\flip_s(\tau^\circ)=\tau$ and such that $d$ is equivalent to $\mu_s^L(d)$ as gradings on $Q(\tau^\circ)$ (through the canonical quiver isomorphism $Q(\tau) \simeq Q(\tau^\circ)$). 

Therefore the graded algebras $\mathcal{P}(Q(\tau),S(\tau),\mu^L_s(d))$ and $\mathcal{P}(Q(\tau^\circ),S(\tau^\circ),d)$ are graded equivalent.
\end{lemma}

\begin{proof}
We start by a piece of notation. For a valency $\geq 2$-tagged triangulation $\sigma$ and a puncture $p$, we denote by $t_p(\sigma)$ the valency $\geq 2$-tagged triangulation obtained from $\sigma$ by changing the tagging at $p$. Notice that if $j$ is an arc of $\sigma$ such that $\flip_j(\sigma)$ is also a valency $\geq 2$-tagged triangulation, then we clearly have $t_p\circ\flip_j(\sigma)=\flip_j\circ t_p(\sigma)$.

We proceed by induction on the number $\ell$ of punctures of $\tau$ that are notched. If $\ell=0$, the result trivially holds. Assume $\ell>0$ and let $p$ be a notched puncture. Let $\tau'$ be a triangulation with the following properties:

\begin{itemize}
\item $\tau'$ is valency $\geq 2$-tagged triangulation;
\item $\tau$ and $\tau'$ have the same tagging at each puncture;
\item $p$ has valency $2$ at $p$.
\end{itemize}

Then there exists a sequence of flips $t=(t_1,\ldots, t_k)$ such that $\flip_t(\tau^\circ)=\tau'^\circ$ and such that any intermediate triangulation is a valency $\geq 2$-ideal triangulation. Since $\flip_t$ commutes with any $t_{q}$, we have $\flip_t(\tau)=\tau'$ and any intermediate triangulation is a valency $\geq 2$-tagged triangulation.

Denote by $i$ and $j$ the arcs incident to $p$ in $\tau'$. Then one has $\flip_i\circ\psi_{i,j}\circ\flip_i(\tau')=t_p(\tau')$ as shown in the following picture.

\[\scalebox{0.9}{\begin{tikzpicture}[scale=0.8,>=stealth]
\draw (0,-2)--node [fill=white, inner sep=2pt]{$i$}(0,0)--node [fill=white, inner sep=2pt]{$j$}(0,2).. controls (1,1.5) and (1,-1.5).. (0,-2).. controls (-1,-1.5) and (-1,1.5)..(0,2);
\node at (0,-2) {$\bullet$};
\node at (0,2) {$\bullet$};
\node at (0,0) {$\bullet$};
\node at (0,0.3) {$\bowtie$};
\node at (0,-0.3) {$\bowtie$};

\draw[<->] (1,0)--node[yshift=10pt]{$\flip_i$} (3,0);

\draw (4,0).. node [fill=white, inner sep=2pt]{$i$} controls (4.2,0.2) and (4.2,1.8)..(4,2);
\draw (4,0).. node [fill=white, inner sep=2pt]{$j$} controls (3.8,0.2) and (3.8,1.8)..(4,2);
\draw(4,2).. controls (5,1.5) and (5,-1.5).. (4,-2).. controls (3,-1.5) and (3,1.5)..(4,2);
\node at (4,-2) {$\bullet$};
\node at (4,2) {$\bullet$};
\node at (4,0) {$\bullet$};
\node at (3.9,0.4) {$\bowtie$};

\draw[<->] (5,0)--node[yshift=10pt]{$\psi_{i,j}$} (7,0);

\draw (8,0).. node [fill=white, inner sep=2pt]{$j$} controls (8.2,0.2) and (8.2,1.8)..(8,2);
\draw (8,0).. node [fill=white, inner sep=2pt]{$i$} controls (7.8,0.2) and (7.8,1.8)..(8,2);
\draw(8,2).. controls (9,1.5) and (9,-1.5).. (8,-2).. controls (7,-1.5) and (7,1.5)..(8,2);
\node at (8,-2) {$\bullet$};
\node at (8,2) {$\bullet$};
\node at (8,0) {$\bullet$};
\node at (7.9,0.4) {$\bowtie$};

\draw[<->] (9,0)--node[yshift=10pt]{$\flip_{i}$} (11,0);

\draw (12,-2)--node [fill=white, inner sep=2pt]{$i$}(12,0)--node [fill=white, inner sep=2pt]{$j$}(12,2).. controls (13,1.5) and (13,-1.5).. (12,-2).. controls (11,-1.5) and (11,1.5)..(12,2);
\node at (12,-2) {$\bullet$};
\node at (12,2) {$\bullet$};
\node at (12,0) {$\bullet$};

\end{tikzpicture}}\]

 Denote by $t^{-}$ the sequence $(t_k,\ldots, t_1)$. Since $t_p$ commutes with $\flip_{t}$, we get $$\flip_{t^{-}}\circ \flip_i\circ \psi_{i,j}\circ\flip_i\circ\flip_t(\tau)=\flip_{t^{-}}\circ \flip_i\circ \psi_{i,j}\circ\flip_i(\tau')=\flip_{t^-}\circ t_p(\tau')=t_p\circ\flip_{t^-}(\tau')=t_p(\tau).$$

Note that if we denote by $t'$ the sequence of flips $t^-=(t_k,\ldots,t_1)$ where each occurence of $i$ is replaced by $j$ and vice versa, then the triangulation $\flip_{t'}\circ \flip_j\circ\flip_i\circ\flip_t(\tau)$ is the triangulation $t_p(\tau)$ where the labeling of the arcs $i$ and $j$ are switched.  

Applying the above sequence of flips to the grading $d$ we obtain 
\begin{align*}
\mu_{t^-}^L\circ \mu_i^L\circ\psi_{i,j}^*\circ\mu_i^L\circ\mu_t^L(d) & =\mu_{t-}^L\mu_i^L((\mu_i^L\mu_t^L(d))\circ\psi_{i,j}) \\
& \underset{\gr}{\sim} \mu_{t^-}^L\mu_i^L (\mu_i^L\mu_t^L(d)) & \textrm{by Proposition \ref{prop:pop-is-quiver-grading-equivalence}}\\
& \underset{\gr}{\sim} d & \textrm{by Remark \ref{rmk:gradedmut} (3) }.
\end{align*}

Now $t_p\triangtau$ is a nice tagged triangulation with $\ell-1$ notched punctures, hence by induction hypothesis there exists a sequence $s$ such that $\flip_s(t_p\tau)=(t_p\tau)^\circ=\triangtau^\circ$  and $\mu_s^L(d)$ is graded equivalent to $d$. Therefore we get the first statement of the lemma.

\medskip
Therefore the algebras  $\mathcal{P}(Q(\tau),S(\tau),d)$ and $\mathcal{P}(Q(\tau),S(\tau),\mu_s^L(d))$ are graded equivalent via the identity. Moreover, as mentioned at the beginining of the proof of Theorem \ref{thm:flip=>graded-right-equivalence}, the right equivalence $(Q(\tau),S(\tau))\to (Q(\tau^\circ),S(\tau^{\circ}))$ is graded for any grading making $S(\tau)$ (hence $S(\tau^\circ)$) homogenous of degree $1$. Thus the graded algebras $\mathcal{P}(Q(\tau),S(\tau),d)$ and $\mathcal{P}(Q(\tau^\circ),S(\tau^\circ),d)$ are isomorphic and we get the second statement.

\end{proof}


\section{Derived Invariants for surface cut algebras}\label{section5}


\subsection{Derived and cluster categories for algebras of global dimension 2}\label{sect::derived}
We start by briefly recalling some definitions and results of \cite{Amiot-gldim2,AO-cl-equiv-der-equiv} that will be used in the proof of the main theorem. For basic notions on $2$-Calabi-Yau categories and cluster-tilting theory we refer to \cite{Amiot-survey}.

For a finite dimensional algebra $\Lambda$ with finite global dimension, we denote by $\mathcal{D}^b(\Lambda)$ the bounded derived category of finitely generated $\Lambda$-modules. We denote by $\mathbb S_2$ the autoequivalence $-\overset{L}\otimes_\Lambda \Hom_k(\Lambda,k)[-2]$ of $\mathcal{D}^b(\Lambda)$. The algebra $\Lambda$ is said to be $\tau_2$-finite if the functor $\tau_2=H^0(\mathbb S _2):\mod \Lambda \to \mod \Lambda$ is nilpotent.

Let $\Lambda$ be a finite dimensional algebra of global dimension $\leq 2$ which is $\tau_2$-finite. Its cluster category  $\mathcal{C}_2(\Lambda)$ is defined in \cite{Amiot-gldim2}. It is a $2$-Calabi-Yau category with cluster-tilting objects. It comes together with a triangle functor $\pi:\mathcal{D}^b(\Lambda)\to \mathcal{C}_2(\Lambda)$ and with a canonical algebra isomorphism $$\End_{\mathcal{C}_2(\Lambda)}(\pi M)\cong \bigoplus_{p\in\mathbb Z}\Hom_{\mathcal{D}^b(\Lambda)}(M,\mathbb{S}_2^{-p}M), \textrm{ for all }M\in \mathcal{D}^b(\Lambda).  $$
Hence each algebra $\End_{\mathcal{C}_2(\Lambda)}(\pi M)$ inherits a natural $\mathbb Z$-grading. Furthermore $\pi\Lambda$ is a cluster-tilting object in $\mathcal{C}_2(\Lambda)$.

Note that if $X$ and $Y$ are indecomposable objects in $\mathcal{D}^b(\Lambda)$ such that $\pi M\simeq \pi N$ in $\mathcal{C}_2(\Lambda)$, then there exists an integer $r$ such that $M\simeq \mathbb S_2^rN$. Therefore, if $M\simeq M_1\oplus\cdots\oplus M_n$ and $N\simeq N_1\oplus\cdots\oplus N_n$ are basic objects in $\mathcal{D}^b(\Lambda)$ such that $\pi M\simeq \pi N$, then (up to renumbering the summands of $N$) there exist integers $r_i$ such that $M_i\simeq \mathbb S^{r_i}_2N_i$, which exactly means that the graded algebras $\End_{\mathcal{C}_2(\Lambda)}(\pi M)$ and $\End_{\mathcal{C}_2(\Lambda)}(\pi N)$ are graded equivalent.

\medskip
For $T$ a basic object in $\mathcal{D}^b(\Lambda)$ such that $\pi T$ is a cluster-tilting object in $\mathcal{C}_2(\Lambda)$ a notion of graded mutation is introduced in \cite{AO-cl-equiv-der-equiv}. This notion coincides with the mutation in $\mathcal{C}_2(\Lambda)$ in the following sense: For any $T_i$ indecomposable summand of $T$ we have $$\pi(\mu_{T_i}^L(T))\cong \mu_{\pi(T_i)}(\pi T).$$ Moreover assume there exists an isomorphism of graded algebra $$\End_{\mathcal{C}_2(\Lambda)}(\pi T)\underset{\mathbb Z}{\cong}\mathcal{P}(Q,S,d)$$ for a certain graded QP $(Q,S,d)$ and assume that there is no loops and no $2$-cycles in $Q$ at the vertex $i$ corresponding with $T_i$, then we have $$\End_{\mathcal{C}_2(\Lambda)}(\pi (\mu_i^LT))\underset{\mathbb Z}{\cong}\mathcal{P}(\mu_i^L(Q,S,d)).$$

\subsection{Surface cut algebras}

\begin{definition}
Let $\triangtau$ be a valency $\geq 2$-triangulation on $\surf$.
\begin{enumerate}
\item An \emph{admissible cut} for $\tau$ is a degree-1 map $d:\Qtau_1\to \mathbb Z$ whose image is contained in $\{0,1\}$ and satisfies $d(a)=0$ for every arrow $a\in\Qtau_1$ not appearing in $\Qtau_2$. 
\item
If $d$ is an admissible cut, then the degree zero subalgebra $\Lambda(\tau,d)$ of $\mathcal{P}(Q(\tau), S(\tau),d)$ is called \emph{surface cut algebra associated with $(\tau,d)$} if its global dimension is $\leq 2$.
\end{enumerate}
\end{definition}

We will see in the next sections that, given a triangulation, it is not always possible to find an admissible cut, and that, given an admissible cut, the corresponding degree zero subalgebra of the Jacobian algebra does not always have global dimension at most two. However, we will show (Corollary \ref{cor:existence-surface-algebra}) that for any surface $\surf$ with non empty boundary, there exists a valency $\geq 3$-triangulation $\tau$ and an admissible cut giving rise to a surface cut algebra.

The following result allows us to use all the results mentioned in the previous section in the special case of surface cut algebras.

\begin{proposition}\label{prop::cutAlgebra}
Let $\surf$ be a marked surface with non-empty boundary.  Let $\Lambda(\tau,d)$ be a surface cut algebra associated to a triangulation $\tau$ and an admissible cut $d$. Then the following properties hold for $\Lambda$:
\begin{itemize}
\item The algebra $\Lambda$ is $\tau_2$-finite;
\item There is an isomorphism of graded algebras $\End_{\mathcal{C}_2(\Lambda)}(\pi \Lambda)\cong \mathcal{P}(Q(\tau), S(\tau),d)$;
\item All cluster-tilting objects in $\mathcal{C}_2(\Lambda)$ can be linked by a sequence of mutations.
\end{itemize}
\end{proposition}

\begin{proof}
The proof is exaclty similar to the proof of Theorem 3.8 in \cite{Amiot-Grimeland}. The fact that all cluster tilting object are related by mutations follows from a combination of \cite[Corollary 3.5]{CKLP} and \cite[Theorem 7.11]{FST}.
\end{proof}

\subsection{Main result}

\begin{theorem}\label{theorem cluster cat}
Let $\surf$ be a surface with non-empty boundary. 
Let $\tau$ and $\tau'$ be valency $\geq 2$-triangulations of $\surf$, with respective admissible cuts $d$ and $d'$.  
Assume that the corresponding cut algebras $\Lambda=\Lambda(\tau,d)$ and $\Lambda'=\Lambda(\tau', d')$ have global dimension $\leq 2$. Then the following statements are equivalent:
\begin{enumerate}
\item there is a triangle equivalence $\mathcal{D}^b(\Lambda)\cong\mathcal{D}^b(\Lambda')$;
\item there exists an orientation preserving homeomorphism $\Phi:\Sigma\to \Sigma$ with $\Phi(\mathbb M)=\mathbb M$ such that for any closed curve $\gamma$, $d(\overline{\gamma}^\tau)=d'(\overline{\Phi(\gamma)}^{\tau'})$.

\end{enumerate}
\end{theorem}

\begin{remark}
Suppose that $\tau'$ is a valency $\geq 2$-triangulation and $\Phi$ is an orientation-preserving homeomorphism such that $\Phi(\mathbb{M})=\mathbb{M}$. Then $\Phi$ induces a quiver isomorphism $Q(\Phi^{-1}(\tau'))\rightarrow Q(\tau')$ that in turn induces an isomorphism of chain complexes $C_\bullet(\Phi^{-1}(\tau'))\rightarrow C_\bullet(\tau')$. We abuse notation and write $\Phi$ to denote the latter chain complex isomorphism. Then we have a group homomorphism $C^1(\tau')\rightarrow C^1(\Phi^{-1}(\tau'))$ given by $d'\mapsto d'\circ\Phi$. Suppose $d'$ is a degree-1 map on $Q(\tau')$; then its image $d'\circ\Phi$ is a degree-1 map on $Q(\Phi^{-1}(\tau'))$. Hence, if we are given another valency $\geq 2$-triangulation $\tau$, then by connecting $\Phi^{-1}(\tau')$ with $\tau$ by a sequence of flips involving only nice valency $\geq 2$-triangulations, then $d'\circ\Phi$ induces a degree-1 map $\{d'\circ\Phi\}_\tau$ on $Q(\tau)$ (again, this degree-1 map depends on the sequence of flips, but any two such sequences induce equivalent gradings on $Q(\tau)$ --provided that only valency $\geq 2$-triangulations are involved in the sequences).

Then combining Lemma \ref{lemma:equivalence-of-gradings-characterization} with Corollary \ref{corollary3.13} we easily deduce that item (2) of Theorem \ref{theorem cluster cat} is equivalent to:
\begin{itemize}
\item[(2')] there exists an orientation preserving homeomorphism $\Phi:\Sigma\to \Sigma$ with $\Phi(\mathbb M)=\mathbb M$ such that $[d-\{d'\circ\Phi\}_\tau]_\tau=0$ as an element in $H^1(C^\bullet(\tau))$.

\item[(2'')] there exists an orientation preserving homeomorphism $\Phi:\Sigma\to \Sigma$ with $\Phi(\mathbb M)=\mathbb M$ such that $[d-\{d'\circ\Phi\}_\tau]_{H^1(\Sigma)}=0$ as an element in $H^1(\Sigma)$.
\end{itemize}

\end{remark}

\begin{proof}[Proof of Theorem \ref{theorem cluster cat}]
The proof here is  similar to the proof of \cite[Thm 3.13]{Amiot-Grimeland}, but slightly more complicated given the presence of punctures, which makes it necessary to keep self-folded triangles and tagged triangulations in mind.

Let us immediately treat the small cases.  If $\surf$ is an unpunctured disc with at most $3$ marked points, then the result is trivial since there are no internal arcs.  If $\surf$ is a once-punctured disc with at most $4$ marked points or a twice-punctured disc with $2$ marked points, then the result is true by \cite[Corollary 3.16]{AO-acyc-cl-type}.

Therefore, for the rest of the proof, we will assume that $\surf$ is not one of the following:
\begin{itemize}
 \item An unpunctured disc with at most $3$ marked points;
 \item a once-punctured disc with at most $4$ marked points on the boundary;
 \item a twice-punctured disc with $2$ marked points on the boundary.
\end{itemize}

$(1)\Rightarrow (2)$ Assume that we have a derived equivalence $F:\mathcal{D}^b(\Lambda) \to \mathcal{D}^b(\Lambda')$.   This equivalence extends to cluster categories, and we get a commutative diagram of functors
$$
\xymatrix{\mathcal{D}^b(\Lambda)\ar[r]^{F}_{\sim}\ar[d]_{\pi} & \mathcal{D}^b(\Lambda')\ar[d]^{\pi'}\\ \mathcal{C}_2(\Lambda)\ar[r]^{f}_{\sim} & \mathcal{C}_2(\Lambda')}.
$$

By a combination of \cite[Theorem 7.11]{FST} and \cite[Corollary 3.5]{CKLP}, there is a bijection $$\xymatrix{ \{ \textrm{cluster-tilting objects in }\mathcal{C}_2(\Lambda')\} \ar@{<->}[rr]^{1-1} &&\{\textrm{tagged triangulations of } \surf\}},$$
compatible with flips and mutations and in which $\pi' \Lambda'$ corresponds to $\tau'$.  Since $f(\pi\Lambda)$ is a cluster-tilting object of $\mathcal{C}_2(\Lambda')$, it corresponds to a tagged triangulation $\tau^{\bowtie}$.  

The two triangulations $\tau'$ and $\tau^{\bowtie}$ are related by a sequence of flips, which we can split into two subsequences $s'$ and $s$ such that
$$
 \flip_{s'}(\tau') = (\tau^{\bowtie})^\circ \quad \textrm{ and } \quad \flip_{s}((\tau^{\bowtie})^\circ) = \tau^{\bowtie}.
$$

We can assume that all intermediate triangulations obtained while applying $s'$ are valency $\geq 2$, and thus ideal. To see this, we first show that $(\tau^{\bowtie})^\circ$ is a valency $\geq 2$-triangulation.  Indeed, the quivers $Q((\tau^{\bowtie})^\circ)$ and 
$Q(\tau^{\bowtie})$ are isomorphic (cf. \cite{FST}), $Q(\tau^{\bowtie})$ is (isomorphic to) the Gabriel 
quiver of $\End_{\mathcal{C}_2(\Lambda')}(\tau^{\bowtie})$, $Q(\tau)$ is (isomorphic to) the Gabriel 
quiver of $\End_{\mathcal{C}_2(\Lambda)}(\tau)$, and
$\End_{\mathcal{C}_2(\Lambda')}(\tau^{\bowtie})\cong\End_{\mathcal{C}_2(\Lambda)}(\tau)$ since 
$f$ is an equivalence of categories. Hence $Q((\tau^{\bowtie})^\circ)$ and $Q(\tau)$ are isomorphic 
quivers, and this implies that $(\tau^{\bowtie})^\circ$ is a valency $\geq 2$-triangulation by Proposition 
\ref{prop::BS}. 

Since both $\tau'$ and $(\tau^{\bowtie})^\circ$ are valency $\geq 2$-triangulations, \cite[Equation 6.4]{LF2} implies the existence of a sequence $s'$ of flips such that $\flip_{s'}(\tau')=(\tau^{\bowtie})^\circ$ and with the property that all the intermediate triangulations of the sequence are valency $\geq 2$-triangulations. 

This also implies that we can take the sequence $s$ to be as in Lemma \ref{lemma:tau-and-tau-circ-equivalent-grading}.

Consider the object $\mu^L_s\mu^L_{s'}\Lambda'$ in $\mathcal{D}^b(\Lambda')$.  Its image by $\pi'$ is $f(\pi\Lambda)$, since $\pi'$ commutes with mutation.  Moreover, since $f(\pi\Lambda)$ is isomorphic to $\pi'F\Lambda$, then by Section \ref{sect::derived}, the algebra $\End_{\mathcal{C}_2(\Lambda')}(\pi'F\Lambda) = \bigoplus_{p\in\mathbb{Z}} \Hom_{\mathcal{D}^b(\Lambda')}(F\Lambda, \mathbb{S}_2^{-p}F\Lambda)$ is graded equivalent to $\End_{\mathcal{C}_2(\Lambda')}(\pi'\mu^L_s\mu^L_{s'}\Lambda') = \bigoplus_{p\in\mathbb{Z}} \Hom_{\mathcal{D}^b(\Lambda')}(\mu^L_s\mu^L_{s'}\Lambda', \mathbb{S}_2^{-p}\mu^L_s\mu^L_{s'}\Lambda')$.  Moreover, by Section \ref{sect::derived},
$$
\End_{\mathcal{C}_2(\Lambda')}(\pi'\mu^L_s\mu^L_{s'}\Lambda') \underset{\mathbb Z}{\cong} \mathcal{P}\big(\mu^L_s\mu^L_{s'}(Q(\tau'), S(\tau'), d')\big).
$$

Now, by Theorem \ref{thm:flip=>graded-right-equivalence}, 
$$
\mathcal{P}\big(\mu^L_s\mu^L_{s'}(Q(\tau'), S(\tau'), d')\big) \underset{\mathbb Z}{\cong} \mathcal{P}\big( Q(\flip_s\flip_{s'}\tau'), S(\flip_s\flip_{s'}\tau'), \mu^L_s\mu^L_{s'}d' \big).
$$
The right hand side of the last equation is equal to $\mathcal{P}\big( Q(\flip_s(\tau^{\bowtie})^\circ), S(\flip_s(\tau^{\bowtie})^\circ), \mu^L_s\mu^L_{s'}d' \big)$, which is in turn graded equivalent to $\mathcal{P}\big( Q((\tau^{\bowtie})^\circ), S((\tau^{\bowtie})^\circ), \mu^L_{s'}d' \big)$ by Lemma \ref{lemma:tau-and-tau-circ-equivalent-grading}.

On the other hand, since $F$ is an equivalence, it commutes with $\mathbb{S}_2$, so
$$
\End_{\mathcal{C}_2(\Lambda')}(\pi' F\Lambda) \underset{\mathbb Z}{\cong} \bigoplus_{p\in \mathbb{Z}} \Hom_{\mathcal{D}^b\Lambda}(\Lambda, \mathbb{S}_2^{-p}\Lambda) = \End_{\mathcal{C}_2(\Lambda)}(\pi\Lambda),
$$
which is in turn isomorphic to $\mathcal{P}(Q(\tau), S(\tau), d)$ by Proposition \ref{prop::cutAlgebra}.

Combining this, we get that $\mathcal{P}(Q(\tau), S(\tau), d)$ is graded equivalent to $\mathcal{P}(Q((\tau^{\bowtie})^\circ), S((\tau^{\bowtie})^\circ), \mu^L_{s'}(d))$.

Thus, by Proposition \ref{fundamental} there exists an orientation preserving homeomorphism $\Phi:\Sigma\to\Sigma$ with $\Phi(\mathbb M)=\mathbb M$ and such that for every $\tau$-admissible closed curve $\gamma$ we have
$$
d(\overline{\gamma}^\tau)=\mu_{s'}^L(d')(\overline{\Phi(\gamma)}^{\flip_{s'}(\tau')}).
$$
Since in the sequence of flips $s'$ all intermediate triangulations are valency $\geq 2$-triangulations, Lemma \ref{lemma:graded-flip-preserves-evaluation} implies that

$$
\mu_{s'}^L(d')(\overline{\Phi(\gamma)}^{\flip_{s'}(\tau')})=
\overline{\Phi(\gamma)}^{\tau'}.
$$

$(2)\Rightarrow(1)$ For this direction, the proof is exactly similar to the one of \cite{Amiot-Grimeland}, we sketch it for the convenience of the reader. Let $\tau''=\Phi^{-1}(\tau')$ and $d''$ be the grading induced by $d'$ on $Q(\tau'')$ via the isomorphism $Q(\tau')\cong Q(\tau'')$. Then we clearly get an isomorphism of graded algebras
$$
\mathcal{P}(Q(\tau'), S(\tau'),d')\underset{\mathbb Z}{\cong}\mathcal{P}(Q(\tau''), S(\tau''),d'').
$$ 
Hence we have an isomorphism $\Lambda'\simeq\Lambda''$ where $\Lambda''$ is the surface cut algebra associated with $(\tau'',d'')$.
Moreover we have for any simple closed curve $\gamma$
$$d''(\overline{\gamma}^{\tau''})=d''(\overline{\gamma}^{\Phi^{-1}\tau'})=d'(\overline{\Phi (\gamma)}^{\tau'}).$$

Now let $s$ be a sequence of flips such that $\mu_s(\tau'')=\tau$ and such that all intermediate triangulations are nice and plain. Then $\mu_s^L(d'')$ and $d$ are both degree-1 maps on $\Qtau$ satisfying
\begin{align*} \mu_s^L(d'')(\overline{\gamma}^\tau) & = d'' (\overline{\gamma}^{\tau''}) & \textrm{by Lemma \ref{lemma:graded-flip-preserves-evaluation}}\\
 & = d'(\overline{\Phi (\gamma))}^{\tau'}) & \textrm{by the above equality}\\
 & = d(\overline{\gamma}^\tau) & \textrm{by hypothesis}
\end{align*}
for any simple closed curve $\gamma$ on $\Sigma$.

Thus $d$ and $\mu_s^L(d'')$ are equivalent gradings, which implies that $\Lambda$ and $\Lambda''$ (and hence $\Lambda'$) are derived equivalent by \cite[Corollary 6.14]{AO-cl-equiv-der-equiv}.

\end{proof}


Theorem \ref{theorem cluster cat} has the following immediate consequence:

\begin{corollary}
All surface cut algebras coming from an arbitrarily punctured polygon are derived equivalent.
\end{corollary}

This result was already known in the case of the disk with one (resp. two) puncture. Indeed, in this case, the surface cut algebra is cluster equivalent to the cluster category of type $D$ (resp. $\tilde{D}$). And since $D$ and $\tilde{D}$ are trees, one can conclude using \cite[Cor. 3.16]{AO-acyc-cl-type}.


\section{Cuts and perfect matchings}\label{section6}


Let $\surf$ be a surface with marked points and non-empty boundary, and suppose that $\triangtau$ is a valency $\geq 3$-triangulation. Then every triangle of $\triangtau$ contains either 0, 1 or 3 arrows of the quiver $\Qtau$. Define a graph $\bipgraph{\triangtau}$ as follows:

\begin{enumerate}
\item The vertex set of $\bipgraph{\triangtau}$ is divided into three types of vertices, white, black and grey:
\begin{itemize}
\item the set $\whiteset$ of white vertices is in bijection with the union of the set of internal triangles of $\triangtau$ with the set of triangles with one side being a boundary segment and one vertex being a puncture;
\item the set $\blackset$ of black vertices is in bijection with the set of punctures of the surface;
\item the set $\greyset$ of grey vertices is in bijection with the union of the set of internal triangles that share at least one point with the boundary with the set of triangles with one side being a boundary segment and one vertex being a puncture.
\end{itemize}
\item the edges of $\bipgraph{\triangtau}$ are only allowed to connect white vertices to black or grey vertices, and they satisfy the following further constraints:
\begin{itemize}
\item for every arrow $\alpha$ shared by a puncture and a triangle, $\bipgraph{\triangtau}$ has an edge $E_\alpha$ connecting the corresponding black and white vertices;
\item for every arrow $\alpha$ in an internal triangle and not shared by a puncture, $\bipgraph{\triangtau}$ has an edge $E_\alpha$ connecting the corresponding grey and white vertices;
\item for every triangle containing exactly one arrow $\alpha$, if this arrow is shared by such triangle and a puncture, then $\bipgraph{\triangtau}$ has an edge $F_\alpha$ connecting the corresponding white and grey vertices;
\item $\bipgraph{\triangtau}$ does not have any other edges besides the ones we have just introduced.
\end{itemize}
\end{enumerate}

\begin{remark}
\begin{itemize}
\item The graph $\bipgraph{\triangtau}$ may be disconnected (see Figure \ref{Fig:Gtau_example1} );
\item for every internal triangle, the corresponding white vertex has valency 3, although the edges emanating from it do not necessarily always go to black vertices;
\item for every triangle with one side being a boundary segment and one vertex being a puncture, the corresponding white vertex has valency 2;
\item it is possible for two vertices of $\bipgraph{\triangtau}$ to be connected by more than one edge;
\item the graph $\bipgraph{\triangtau}$ is bipartite: a bipartition is given by the white vertex set and the union of the black vertex set and the grey vertex set.
\end{itemize}
\end{remark}

\begin{example} In Figure \ref{Fig:Gtau_example1} we can see a triangulation $\triangtau$ of a twice-punctured octogon as well as the associated graph $\bipgraph{\triangtau}$.
\begin{figure}[!h]
                \caption{Triangulation $\triangtau$ and associated graph $\bipgraph{\triangtau}$}
  \label{Fig:Gtau_example1}
 \[\scalebox{0.8}{
\begin{tikzpicture}[>=stealth,scale=0.7]

\draw (0,0)--(1,3)--(4,4)--(7,3)--(8,0)--(7,-3)--(4,-4)--(1,-3)--(0,0);
\node at (5,0) {$\bullet$};
\node at (3,2) {$\bullet$};
\draw (0,0)--(4,-4)--(1,3)..controls (2,1) and (3,0)..(5,0)--(4,-4)..controls (7,-2) and (7,2)..(4,4)--(5,0)--(3,2)--(4,4).. controls (5,3.8) and (8,0)..(8,0)..controls (8,0) and (5,-3.8)..(4,-4);
\draw (1,3)--(3,2);

\draw[dotted,gray] (10,0)--(11,3)--(14,4)--(17,3)--(18,0)--(17,-3)--(14,-4)--(11,-3)--(10,0);
\node at (15,0) {$\bullet$};
\node at (13,2) {$\bullet$};
\draw[dotted,gray] (10,0)--(14,-4)--(11,3)..controls (12,1) and (13,0)..(15,0)--(14,-4)..controls (17,-2) and (17,2)..(14,4)--(15,0)--(13,2)--(14,4).. controls (15,3.8) and (18,0)..(18,0)..controls (18,0) and (15,-3.8)..(14,-4);
\draw[dotted,gray] (11,3)--(13,2);

\draw[thick] (12.5,3)--(13,2)--(14,2)--(15,0)--(13,1)--(13,2);
\draw[thick] (14,-1)--(15,0)--(15.5,-1);
\draw[thick,double] (13,0)--(14,-1);
\draw[thick] (16.4,-0.9)--(16.4,0.9);
\draw[thick] (16.5,-0.9)--(16.5,0.9);
\draw[thick] (16.6,-0.9)--(16.6,0.9);
\draw[thick,double] (15.5,-1)--(15.5,1);
\draw[very thick, dotted] (12,4)--(12.5,3);
\draw[thick] (14,3)--(14,2);
\draw[thick] (12,2)--(13,1);

\node[fill=white, inner sep=0pt] at (14,-1) {$\circ$};
\node[fill=white, inner sep=0pt] at (13,1) {$\circ$};
\node[fill=white, inner sep=0pt]  at (12.5,3) {$\circ$};
\node[fill=white, inner sep=0pt]  at (14,2) {$\circ$};
\node[fill=white, inner sep=0pt]  at (15.5, -1) {$\circ$};
\node[fill=white, inner sep=0pt]  at (16.5,-1) {$\circ$};

\node[gray] at (15.5,1) {$\bullet$};
\node[gray] at (16.5,1) {$\bullet$};
\node[gray] at (13,0) {$\bullet$};
\node[gray] at (12,2) {$\bullet$};
\node[gray] at (12,4) {$\bullet$};
\node[gray] at (14,3) {$\bullet$};

\end{tikzpicture}}\]  
        \end{figure}
Edges of type $E_\alpha$ have been drawn undotted, while the edge of type $F_\alpha$ has been drawn dotted.
\end{example}

\begin{lemma} The set of cuts of $\Qtau$ is in bijection with the set $\{M\suchthat M$ is a partial matching on $\bipgraph{\triangtau}$ and there exists a subset $V_M\subseteq\greyset$ such that $M$ is a perfect matching on $\bipgraph{\triangtau}\setminus V_M\}$.
\end{lemma}

\begin{proof}
Let $d$ be an admissible cut of $\Qtau$, then $d$ is entirely defined by the set $C:=\{\alpha\in Q(\tau)_1 | d(\alpha)=1\}$. Define the subset $M=M(C)$ of the edge set of $\bipgraph{\triangtau}$ as follows:
\[M(C):=\{E_\alpha\suchthat \alpha\in C\}\cup\{F_\alpha\suchthat \alpha\notin C\textrm{ and }F_\alpha\textrm{ is defined}\}.\]

Note that is $\alpha$ is an arrow of $Q(\tau)$ such that $E_\alpha$ is not defined, then $\alpha$ belongs to a triangle with one side being a boundary segment and the three vertices on the boundary. Then $\alpha$ does not appear in $Q(\tau)_2$, and thus cannot be in $C$.
 
We claim that $M$ is a  partial matching on $\bipgraph{\triangtau}$. To prove this claim, we first show that every vertex which is white or black belongs exactly to one edge in $M$. It is obvious that every black vertex belongs to exactly one edge in $M$. Let $w\in \bipgraph{\triangtau}$ be a white vertex. If it has valency three then  the corresponding triangle $\triang$ of $\triangtau$ is internal, and so contains exactly three arrows $\alpha$, $\beta$ and $\gamma$.  Exactly one of these arrows, say $\alpha$, belongs to $C$, and hence $E_\alpha\in M$ and $E_\beta,E_\gamma\notin M$. If $w$ has valency two, then the corresponding triangle contains one boundary segment and one puncture, hence it contains exaclty one arrow $\alpha$. Then either $\alpha\in C$, in which case $E_\alpha\in M$ and $F_\alpha\notin M$ or  $\alpha\notin C$, in which case $E_\alpha\notin M$ and $F_\alpha\in M$.

Now we show that every grey vertex does not belong to more than one edge in $M$.
Let $g\in \bipgraph{\triangtau}$ be a grey vertex, and let $\triang$ be the corresponding triangle of $\triangtau$. Let $w$ be the white vertex of $\bipgraph{\triangtau}$ corresponding to $\triang$. It is clear that no edge of $\bipgraph{\triangtau}$ connects $g$ with a vertex different from $w$, and that $w$ is connected to $g$ by at least one edge of $\bipgraph{\triangtau}$. If $\triang$ contains exactly one arrow of $\Qtau$, then $w$ and $g$ are connected by exactly one edge of $\bipgraph{\triangtau}$ and it is hence obvious that $g$ belongs to at most one edge in $M$. If $\triang$ contains exactly three arrows, say $\alpha$, $\beta$ and $\gamma$, then exactly one of these three arrows, say $\alpha$, belongs to $C$, and then it is clear that both of the following containments are impossible:
$
E_\beta\in M, \ E_{\gamma}\in M,
$
which shows that $g$ belongs to at most one edge in $M$.

We conclude that $M$ is a partial matching on $\bipgraph{\triangtau}$. Since we have shown that every white vertex and every black vertex belong to at least one edge in $M$, the existence of a subset $V_M$ of the grey vertex set such that $M$ is a perfect matching on $\bipgraph{\triangtau}\setminus V_M$ is obvious.

\medskip
Conversely, suppose that $M$ is a partial matching on $\bipgraph{\triangtau}$ and that there exists a subset $V_M$ of the grey vertex set such that $M$ is a perfect matching on $\bipgraph{\triangtau}\setminus V_M$. Define $C=C(M)$ defined as follows.
\[C(M)=\{\alpha\suchthat E_\alpha\in M\}.\]
It is obvious that $C$ is a cut of $\QStau$.

We leave in the hands of the reader the verification that the assignments $M\mapsto C(M)$ and $C\mapsto M(C)$ are inverse to each other.
\end{proof}

\begin{remark} Let $M$ be a partial matching on $\bipgraph{\triangtau}$ for which there exists a subset $V_M$ of the grey vertex set with the property that $M$ is a perfect matching on $\bipgraph{\triangtau}\setminus V_M$. Note that the subset $\{E\in M \ | E=E_\alpha\}$ of $M$ determines $M$ uniquely.

\end{remark}

\begin{example} In Figure \ref{Fig:Gtau_cut_example1} we can see a triangulation $\tau$ of a twice-punctured octogon, a cut of the QP $\QStau$ (arrows of degree 1 are dotted), and the partial matching corresponding to the cut.

\begin{figure}[!h]
                \caption{Admissible cut of $\triangtau$ and associated perfect matching of  $\bipgraph{\triangtau}$}
  \label{Fig:Gtau_cut_example1}
 \[\scalebox{0.8}{
\begin{tikzpicture}[>=stealth,scale=0.7]

\draw (0,0)--(1,3)--(4,4)--(7,3)--(8,0)--(7,-3)--(4,-4)--(1,-3)--(0,0);
\node at (5,0) {$\bullet$};
\node at (3,2) {$\bullet$};
\draw (0,0)--(4,-4)--(1,3)..controls (2,1) and (3,0)..(5,0)--(4,-4)..controls (7,-2) and (7,2)..(4,4)--(5,0)--(3,2)--(4,4).. controls (5,3.8) and (8,0)..(8,0)..controls (8,0) and (5,-3.8)..(4,-4);
\draw (1,3)--(3,2);

\node (A) at (2,-2) {};
\node (B) at (2.7,-1) {};
\node (C) at (2,2.5) {};
\node (D) at (3,0.5) {};
\node (E) at (3.5,3) {};
\node (F) at (4,1) {};
\node (G) at (4.5,-2) {};
\node (H) at (4.5,2) {};
\node (I) at (6.3,0) {};
\node (J) at (7,1.3) {};
\node (K) at (7,-1.3) {};

\draw[thick,->] (B)--(A);\draw[thick,->] (B)--(D);\draw[thick,->] (D)--(C);\draw[thick,->] (C)--(F);\draw[thick,->] (E)--(C);\draw[thick,->] (E)--(H);\draw[thick,->] (H)--(F);\draw[thick,->] (H)..controls (5.5,2) and (6,0.5)..(I);\draw[thick,->] (G)..controls (5.5,-1) and (5.5,1)..(H);\draw[thick,->] (I)--(J);
\draw[thick,->] (J)--(K);\draw[thick,->] (D)--(G);

\draw[thick, dotted] (B)--(G); \draw[thick, dotted] (D)--(F); \draw[thick, dotted] (E)--(F); \draw[thick, dotted] (I)--(G); \draw[thick, dotted] (I)--(K);

\draw[dotted,gray] (10,0)--(11,3)--(14,4)--(17,3)--(18,0)--(17,-3)--(14,-4)--(11,-3)--(10,0);
\node at (15,0) {$\bullet$};
\node at (13,2) {$\bullet$};
\draw[dotted,gray] (10,0)--(14,-4)--(11,3)..controls (12,1) and (13,0)..(15,0)--(14,-4)..controls (17,-2) and (17,2)..(14,4)--(15,0)--(13,2)--(14,4).. controls (15,3.8) and (18,0)..(18,0)..controls (18,0) and (15,-3.8)..(14,-4);
\draw[dotted,gray] (11,3)--(13,2);

\draw[very thick] (13,2)--(14,2);
\draw[very thick](15,0)--(13,1);

\draw[very thick] (13,0)--(14,-1);

\draw[very thick] (16.5,-0.9)--(16.5,0.9);

\draw[very thick] (15.5,-1)--(15.5,1);
\draw[very thick] (12,4)--(12.5,3);

\node[fill=white, inner sep=0pt] at (14,-1) {$\circ$};
\node[fill=white, inner sep=0pt] at (13,1) {$\circ$};
\node[fill=white, inner sep=0pt]  at (12.5,3) {$\circ$};
\node[fill=white, inner sep=0pt]  at (14,2) {$\circ$};
\node[fill=white, inner sep=0pt]  at (15.5, -1) {$\circ$};
\node[fill=white, inner sep=0pt]  at (16.5,-1) {$\circ$};

\node[gray] at (15.5,1) {$\bullet$};
\node[gray] at (16.5,1) {$\bullet$};
\node[gray] at (13,0) {$\bullet$};
\node[gray] at (12,2) {$\bullet$};
\node[gray] at (12,4) {$\bullet$};
\node[gray] at (14,3) {$\bullet$};

\end{tikzpicture}}\]  

        \end{figure}

\end{example}

We recall from \cite{Hall} the following criterion for a general bipartite graph to have at least one perfect matching:

\begin{theorem}[Hall's marriage theorem]\label{thm:criterion-for-perfect-matchings} Let $G$ be a bipartite graph. A perfect matching on $G$ exists if and only if every collection of white (resp. black) vertices is edge-connected to at least as many black (resp. white) vertices. In particular, if $G$ admits a perfect matching, then its number of white vertices equals its number of black vertices.
\end{theorem}

Since grey vertices are somehow disposable when it comes to partial matchings on $\bipgraph{\triangtau}$ that correspond to cuts of $\QStau$, we give the following slight modification of Theorem \ref{thm:criterion-for-perfect-matchings}:

\begin{theorem}\label{thm:criterion-for-good-partial-matchings} The graph $\bipgraph{\triangtau}$ admits a partial matching that saturates all white and black vertices if and only if the following two conditions are simultaneously satisfied:
\begin{enumerate}
\item\label{item:white<black-and-grey} For any subset $S_1$ of $\whiteset$ there are at least $|S_1|$ elements of $\blackset\cup\greyset$ that are edge-connected to elements of $S_1$;
\item\label{item:black<white} for any subset $S_2$ of $\blackset$ there are at least $|S_2|$ elements of $\whiteset$ that are edge-connected to elements of $S_2$.
\end{enumerate}
In particular, if a partial matching exists that saturates all black and white vertices, then we have $|\mathbb{B}|\leq |\mathbb{W}|\leq |\mathbb{B}|+|\mathbb{G}|$.
\end{theorem}

\begin{proof} Necessity is obvious, we prove sufficiency. From \eqref{item:white<black-and-grey} we deduce that there exists a partial matching $M_1$ on $\bipgraph{\triangtau}$ that saturates $\whiteset$, whereas from \eqref{item:black<white} we deduce that there exists a partial matching $M_2$ on $\bipgraph{\triangtau}$ that saturates $\blackset$. If $M_2\subseteq M_1$, then $M_1$ is a partial matching that saturates all white and black vertices and the theorem follows. So, let us assume that $M_2\setminus M_1\neq \varnothing$ and write $D=(M_1\cup M_2)\setminus(M_1\cap M_2)$. Note that $M_2\setminus M_1\subseteq D$.

As a graph, $D$ is the union of its connected components. Any such component $C$ either contains a grey vertex or it does not. If it does, then it is a line\footnote{that is, isomorphic, as a graph, to a Dynkin diagram of type $A$}, it contains exactly one grey vertex and this grey vertex is a leaf (see Figure \ref{Fig:grey_leaf_in_graph_component}). If $C$ does not contain any grey vertex, then $C$ is a cycle.
\begin{figure}[!h]
                \caption{If $C$ contains a grey vertex, then $C$ is a line and the grey vertex is a leaf of $C$.}\label{Fig:grey_leaf_in_graph_component}
                \centering
                \includegraphics[scale=1]{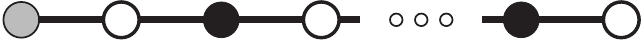}
        \end{figure}

If at least one connected component $C$ of $D$ is either a cycle or a line having a black vertex as a leaf, define $M_1'=(M_1\setminus C)\cup (M_2\cap C)$ and $M_2'=M_2$. Then $M_1'$ is a partial matching on $\bipgraph{\triangtau}$ and it saturates $\whiteset$. Furthermore, the symmetric difference $(M_1'\cup M_2')\setminus (M_1'\cap M_2')$ is properly contained in $D$.

If every connected component of $D$ is a line that has a white vertex as a leaf, then:
\begin{itemize}\item If at least one such component, say $C$, has at least two edges, let $e$ be the unique edge in $C$ that is incident to a grey vertex and define $M_1'=M_1$ and $M_2'=(M_2\setminus C)\cup ((M_1\cap C)\setminus\{e\})$. Then $M_2'$ is a partial matching on $\bipgraph{\triangtau}$ and it saturates $\blackset$. Furthermore, the symmetric difference $(M_1'\cup M_2')\setminus (M_1'\cap M_2')$ is properly contained in $D$.
\item If all components of $D$ are singletons, then at least one of them, say $C$, is contained in a connected component of $\bipgraph{\triangtau}$ that has a black vertex, for otherwise we would have $M_2\subseteq M_1$, against our assumption that $M_2\setminus M_1\neq \varnothing$. Let $e_0$ be the unique element of $C$. Then $e_0$ connects a vertex $g_0\in\greyset$ with a vertex $w_0\in\whiteset$ and there exists an edge $e_1$ that connects $w_0$ with a black vertex $b_0$. Since $e_0\in M_1$, the edge $e_1$ does not belong to $M_1$. It does not belong to $M_2$ either, for otherwise $C=\{e_0\}$ would not be a connected component of $D$. Therefore, $e_1\notin M_1\cup M_2$. Let $e_2$ be the unique element of $M_2$ containing $b_0$. Then $e_2\in M_1$, for otherwise $e_2$ would be contained in a connected component of $D$ that would not be a singleton with a white vertex as a leaf. Set $M_1'=(M_1\setminus\{e_0,e_2\})\cup\{e_1\}$ and $M_2'=(M_2\setminus\{e_2\})\cup\{e_1\}$. Then $M_1'$ (resp. $M_2'$) is a partial matching on $\bipgraph{\triangtau}$ that saturates $\whiteset$ (resp. $\blackset$), and the symmetric difference $(M_1'\cup M_2')\setminus (M_1'\cap M_2')$ has less elements than $D$.
\end{itemize}

We conclude that a partial matching on $\bipgraph{\triangtau}$ indeed exists that simultaneously saturates all white and black vertices.
\end{proof}

%

\begin{corollary}\label{coro::nocut} If $\surf$ is a surface with empty boundary different from a sphere with less than five punctures, and if $\triangtau$ is a valency $\geq 3$-triangulation, then $\QStau$ does not admit any cut whatsoever.
\end{corollary}

\begin{proof} Using induction on the number of punctures of $\surf$, it is easy to prove that $|\mathbb{B}|< |\mathbb{W}|$. The result follows then from Theorem \ref{thm:criterion-for-good-partial-matchings} since $\mathbb{G}=\varnothing$.
\end{proof}

\begin{proposition}\label{prop::nocut} Let $\surf$ be a surface with non-empty boundary and at least one puncture. If the genus of $\surface$ is positive, then there exist ideal triangulations of $\surf$ whose associated QPs do not admit any cuts whatsoever.
\end{proposition}

\begin{proof} It suffices to show the existence of a triangulation $\triangtau$ with the property that for some puncture $p$ and some pair $\triang$, $\triang'$, of different triangles of $\triangtau$, all the sides of $\triang$ and all the sides of $\triang'$ are loops based at $p$. Such a triangulation $\triangtau$ can be constructed as follows. Consider the triangulation $\triangsigma$ shown in Figure \ref{Fig:torus_1bound_1punct}, where the underlying surface $\surfprime$ is the once-punctured torus with exactly one boundary component and exactly one marked point on such component.
\begin{figure}[!h]
                \caption{Triangulation without cuts}\label{Fig:torus_1bound_1punct}
                \centering
                \includegraphics[scale=0.6]{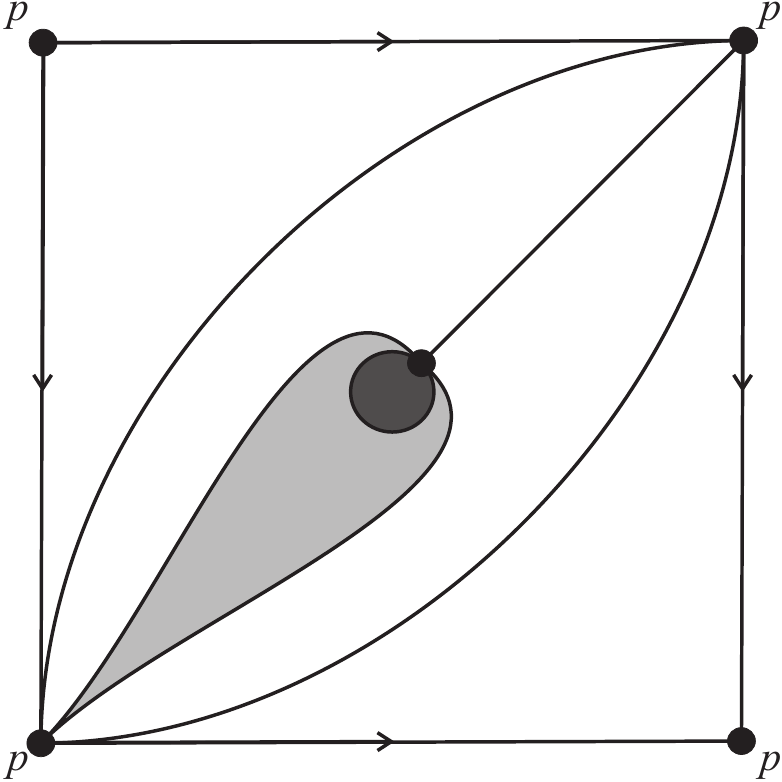}
        \end{figure}
If $\surf=\surfprime$, we take $\triangtau$ to be $\triangsigma$. Otherwise, it is clear that $\surf$ can be obtained from $\surfprime$ by gluing a surface $\surface''$ of genus $g(\surface)-1$ to $\surfaceprime$ along the boundary of a disc cut out in the interior of the unique non-internal triangle of $\triangsigma$ (which has been shadowed in Figure \ref{Fig:torus_1bound_1punct}), and by suitably adding marked points to $\surface''$ and/or to the unique boundary component of $\surfaceprime$ if necessary. Completing $\sigma$ to a triangulation of $\surf$ yields a triangulation $\triangtau$ whose associated QP clearly lacks cuts.
\end{proof}


\section{Cuts and global dimension $\leq 2$}\label{section7}


The aim of this section is to describe the cuts which give rise to algebras of global dimension $\leq 2$.

Let $\surf$ be an oriented surface with marked points and with non empty boundary. Let $\triangtau$ be an ideal valency $\geq 3$-triangulation of $\surf$. Let $d$ be an admissible cut on $\triangtau$. We denote by $\Lambda$ the degree zero subalgebra of the corresponding Jacobian algebra.  The Gabriel quiver $Q(\Lambda)$ of $\Lambda$ can be obtained from $Q(\tau)$ by deleting the arrows that have degree 1. The ideal of relations in $\Lambda$ is generated by the set $\{\partial_\alpha S(\tau), d(\alpha)=1\}$.

We start this section with a result (Lemma \ref{lemma:hook=zero}) that gives some zero relations between the arrows from $\Lambda$. This allows us to describe explicitely all the shapes that a projective indecomposable module can take (Proposition~\ref{prop:shape-projectives}). From this description, we deduce for which configuration a simple has projective dimension $\geq 3$ (Proposition \ref{prop:pd(S_i)>2}). We end the section proving that there exist surface cut algebras associated to almost any surface (Corollary~\ref{cor:existence-surface-algebra}).

\subsection{Zero relations in $\Lambda$}

For arrows $\alpha$, $\beta$ of $\Qtau$, we say that $\alpha\beta$ is a \emph{hook} if the composition $\alpha\beta$ is a path and $\alpha$ and $\beta$ belong to the same triangle of $\triangtau$.

\begin{lemma}\label{lemma:hook=zero}
If $\alpha\beta$ is a hook in $\Lambda$, then for any arrow $\gamma$ of the Gabriel quiver of $\Lambda$, the products $\alpha\beta\gamma$ and $\gamma\alpha\beta$ are zero in $\Lambda$.
\end{lemma}

\begin{proof}
We only prove that $\alpha\beta\gamma$ is zero in $\Lambda$. The proof that the other product is zero as well is similar.

Without loss of generality we can clearly assume that $\alpha\beta\gamma$ is an actual path, that is, that the head of $\gamma$ is the tail of $\beta$. Furthermore, it clearly suffices to show that $\alpha\beta\gamma$ belongs to the Jacobian ideal $J(S(\tau))$.

Let $(\widetilde{\Sigma},\widetilde{\mathbb{M}})$ be a surface with empty boundary, and $\widetilde{\tau}$ be an ideal triangulation of $(\widetilde{\Sigma},\widetilde{\mathbb{M}})$, with the following properties:
\begin{itemize}
\item $\Sigma\subseteq\widetilde{\Sigma}$, $\mathbb{M}\subseteq\widetilde{\mathbb{M}}$ and $\tau\subseteq\widetilde{\tau}$;
\item every puncture of $(\widetilde{\Sigma},\widetilde{\mathbb{M}})$ is incident to at least three arcs in $\widetilde{\tau}$;
\item $(\widetilde{\Sigma},\widetilde{\mathbb{M}})$ is not a sphere with exactly four punctures.
\end{itemize}
One can construct $(\widetilde{\Sigma},\widetilde{\mathbb{M}})$ and $\widetilde{\tau}$ by gluing suitable, possibly punctured, triangulated polygons to $(\Sigma,\mathbb{M})$ along its boundary components.
In what follows, all the arrows are meant to be arrows of $Q(\widetilde{\tau})$.

For any arrow $\eta$ of $Q(\widetilde{\tau})$ there exists exactly one arrow $\mathfrak{t}(\eta)\neq\eta$ that has the same tail as $\eta$, and exactly one arrow $\mathfrak{h}(\eta)$ that has the same head as $\beta$.

Because of the surjective algebra homomorphism $\mathcal{P}(Q(\widetilde{\tau},S(\widetilde{\tau})))\rightarrow\mathcal{P}(Q(\tau),S(\tau))$, in order to prove that $\alpha\beta\gamma$ belongs to $J(S(\tau))$ it is enough to show that $\alpha\beta\gamma$ belongs\footnote{That $\alpha\beta\gamma\in J(S(\widetilde{\tau}))$ has been proved in \cite[Lemma 3.10]{Ladkani} under a hypothesis which is less general than assuming that every puncture of $(\widetilde{\Sigma},\widetilde{\mathbb{M}})$ has valency at least 3.} to the Jacobian ideal $J(S(\widetilde{\tau}))$. Furthermore, to prove that $\alpha\beta\gamma$ belongs to the Jacobian ideal $J(S(\widetilde{\tau}))$, it is enough to show that there exist paths $p_1$ and $p_2$ with the property that the length of the path $p_1\alpha\beta\gamma p_2$ is greater than 3 and, moreover, $\alpha\beta\gamma-p_1\alpha\beta\gamma p_2\in J(S(\widetilde{\tau}))$.
By the previous paragraph, $\beta$ uniquely determines a sequence of arrows $(\beta_n)_{n\geq 0}$ such that
\begin{enumerate}
\item $\beta_0=\beta$ and $\beta_1=\mathfrak{t}(\beta_0)$;
\item for all $n\geq 0$,
$
\beta_{n+2}=
\begin{cases}
\mathfrak{h}(\beta_{n+1}) & \text{if $\beta_{n+1}=\mathfrak{t}(\beta_n)$;}\\
\mathfrak{t}(\beta_{n+1}) & \text{if $\beta_{n+1}=\mathfrak{h}(\beta_n)$.}
\end{cases}
$
\end{enumerate}

Note that we always have $\beta_1\neq \beta_0$. Furthermore, since $Q(\widetilde{\tau})$ has only finitely many arrows, there exists an $n\geq 0$ such that
$\beta_{n+1}=\beta_0$ and $|\{\beta_0,\ldots,\beta_n\}|=n+1$. This integer $n$ is always odd.

For each $m\in\{0,\ldots,n\}$,
\begin{itemize}
\item if $m$ is even, let $\alpha_m$ be the unique arrow such that $\alpha_m\beta_m$ is a hook, and let $\delta_m$ be the unique arrow such that $\beta_m\delta_m$ is a hook;
\item if $m$ is odd, let $\alpha_m$ be the unique arrow such that $\beta_m\alpha_m$ is a hook, and let $\delta_m$ be the unique arrow such that $\delta_m\beta_m$ is a hook.
\end{itemize}
Note that $\alpha_0=\alpha$ and $\alpha_1=\gamma$.

There certainly exist paths $\lambda_0,\ldots,\lambda_n$, such that
$$
\partial_{\delta_m}(S(\widetilde{\tau})) =
\begin{cases}
\alpha_m\beta_m-\lambda_m\alpha_{m+2}\beta_{m+1} & \text{if $m$ is even;}\\
\beta_m\alpha_m-\beta_{m+1}\alpha_{m+2}\lambda_m & \text{if $m$ is odd;}
\end{cases}
$$
which can be written as the following sequence of equalities:
\begin{eqnarray*}
\alpha_0\beta_0 &=& \partial_{\delta_0}(S(\widetilde{\tau}))+ \lambda_0\alpha_{2}\beta_{1} \\
\beta_1\alpha_1 &=& \partial_{\delta_1}(S(\widetilde{\tau}))+\beta_{2}\alpha_{3}\lambda_1\\
\alpha_2\beta_2 & = & \partial_{\delta_2}(S(\widetilde{\tau}))+\lambda_2\alpha_{4}\beta_{3}\\
\beta_3\alpha_3 &=& \partial_{\delta_3}(S(\widetilde{\tau}))+\beta_{4}\alpha_{5}\lambda_3\\
\alpha_4\beta_4 & = & \partial_{\delta_4}(S(\widetilde{\tau}))+\lambda_4\alpha_{6}\beta_{5}\\
\beta_5\alpha_5 &=& \partial_{\delta_5}(S(\widetilde{\tau}))+\beta_{6}\alpha_{7}\lambda_5\\
\alpha_6\beta_6 & = & \partial_{\delta_6}(S(\widetilde{\tau}))+\lambda_6\alpha_{8}\beta_{7}\\
\beta_7\alpha_7 &=& \partial_{\delta_7}(S(\widetilde{\tau}))+\beta_{8}\alpha_{9}\lambda_7\\
\vdots & \vdots & \vdots\\
\alpha_{n-1}\beta_{n-1}&=&\partial_{\delta_{n-1}}(S(\widetilde{\tau}))+\lambda_{n-1}\alpha_{0}\beta_{n}\\
\beta_n\alpha_n&=&\partial_{\delta_n}(S(\widetilde{\tau}))+\beta_{0}\alpha_{1}\lambda_n.
\end{eqnarray*}
Using this sequence of equalities we see that
\begin{eqnarray*}
\alpha\beta\gamma = \alpha_0\beta_0\alpha_1 &\equiv&  \lambda_0\alpha_{2}\beta_{1}\alpha_1 \\
&\equiv& \lambda_0\alpha_{2}\beta_{2}\alpha_{3}\lambda_1\\
&\equiv& \lambda_0\lambda_2\alpha_{4}\beta_{3}\alpha_{3}\lambda_1\\
&\equiv& \lambda_0\lambda_2\alpha_{4}\beta_{4}\alpha_{5}\lambda_3\lambda_1\\
&\equiv& \lambda_0\lambda_2\lambda_4\alpha_{6}\beta_{5}\alpha_{5}\lambda_3\lambda_1\\
&\equiv& \lambda_0\lambda_2\lambda_4\alpha_{6}\beta_{6}\alpha_{7}\lambda_5\lambda_3\lambda_1\\
&\equiv& \lambda_0\lambda_2\lambda_4\lambda_6\alpha_{8}\beta_{7}\alpha_{7}\lambda_5\lambda_3\lambda_1\\
&\equiv& \lambda_0\lambda_2\lambda_4\lambda_6\alpha_{8}\beta_{8}\alpha_{9}\lambda_7\lambda_5\lambda_3\lambda_1\\
\vdots & \vdots & \vdots\\
&\equiv& \lambda_0\lambda_2\lambda_4\lambda_6\ldots\lambda_{n-1}\alpha_{0}\beta_{n}\alpha_{n}\lambda_{n-2}\ldots\lambda_7\lambda_5\lambda_3\lambda_1\\
&\equiv& \lambda_0\lambda_2\lambda_4\lambda_6\ldots\lambda_{n-1}\alpha_{0}\beta_{0}\alpha_{1}\lambda_n\lambda_{n-2}\ldots\lambda_7\lambda_5\lambda_3\lambda_1\\
&=& p_1 \alpha\beta\gamma p_2,
\end{eqnarray*}
where $x\equiv y$ means that $x-y\in J(S(\tau))$.  That $p_1 \alpha\beta\gamma p_2$ has length greater than 3 follows from the fact that for some $m$, the puncture associated to $\delta_m$ has valency at least 4 in $\widetilde{\tau}$, a fact which itself follows from the easy-to-prove fact that, since $(\widetilde{\Sigma},\widetilde{\mathbb{M}})$ is not a sphere with exactly 4 punctures, if a puncture $p$ has valency 3 in $\widetilde{\tau}$, then at most one of the arcs it is connected to by means of arcs in $\widetilde{\tau}$ has valency 3 in $\widetilde{\tau}$.
\end{proof}

\subsection{Projective indecomposable modules}

Let $i$ be an (internal) arc of $\triangtau$.  We denote $\Delta$ and $\Delta'$ the triangles  that contain $i$. We adopt the following notations and orientations for the arcs that are sides of $\Delta$ and $\Delta'$, and for the quiver $Q(\tau)$ around $i$:

 \[\scalebox{0.8}{
\begin{tikzpicture}[>=stealth,scale=1.2]
\draw (1,0) node[below] {$i$};
\draw (1.5,0.75) node[right] {$k$};
\draw (0.5,0.75) node[left] {$j$};
\draw (1.5,-0.75) node[right] {$j'$};
\draw (0.5,-0.75) node[left] {$k'$};
\draw (0,0) node[left] {$p$};
\draw (2,0) node[right] {$q$};
\draw [->](0,0)--(1,0);
\draw (1,0)--(2,0);
\draw[->](2,0)--(1.5,0.75);
\draw(1.5,0.75)--(1,1.5);
\draw[->](1,1.5)--(0.5,0.75);
\draw (0.5,0.75)--(0,0);
\draw(2,0)--(1.5,-0.75);
\draw[<-](1.5,-0.75)--(1,-1.5);
\draw(1,-1.5)--(0.5,-0.75);
\draw [<-](0.5,-0.75)--(0,0);

\draw (4,0)--(6,0)--(5,1.5)--(4,0)--(5,-1.5)--(6,0);
\draw[blue,thick,->] (5,0)--node [fill=white, inner sep=1pt]{$\alpha$}(4.5,0.75);
\draw[blue,thick,->] (4.5,0.75)--node [fill=white, inner sep=1pt]{$\beta$}(5.5,0.75);
\draw[blue,thick,->] (5.5,0.75)--node [fill=white, inner sep=1pt]{$\gamma$}(5,0);
\draw[blue,thick,->] (5,0)--node [fill=white, inner sep=1pt]{$\alpha'$}(5.5,-0.75);
\draw[blue,thick,->] (5.5,-0.75)--node [fill=white, inner sep=1pt]{$\beta'$}(4.5,-0.75);
\draw[blue,thick,->] (4.5,-0.75)--node [fill=white, inner sep=1pt]{$\gamma'$}(5,0);
\draw[blue,thick,->] (4.5,0.75)--node [fill=white, inner sep=1pt]{$\alpha_2$}(3.5,0.75);
\draw[blue,thick,->] (3.5,0.75)--node [fill=white, inner sep=1pt]{$\alpha_3$}(3,0);
\draw[blue,thick,dotted] (3,0)--(3,-0.5);
\draw[blue,thick,->] (5.5,-0.75)--node [fill=white, inner sep=1pt]{$\alpha'_2$}(6.5,-0.75);
\draw[blue,thick,->] (6.5,-0.75)--node [fill=white, inner sep=1pt]{$\alpha'_3$}(7,0);
\draw[blue,thick,dotted] (7,0)--(7,0.5);

\end{tikzpicture}}
\]
Note that there might be identification between vertices or arcs in this picture. But since the valency of each puncture is at least $3$ the two triangles $\Delta$ and $\Delta'$ must be different.
The arrows $\alpha$, $\alpha'$, $\beta$, $\beta'$, $\gamma$ and $\gamma'$ may or may not exist in the quiver $Q(\tau)$ depending wether $j$, $k$, $j'$ or $k'$ are internal arcs, and if they exists they may or may not be arrows in $Q(\Lambda)$ depending on the cut $d$.  We also denote by $\alpha_1=\alpha,\alpha_2,\alpha_3\ldots$ (resp. $\alpha'_1=\alpha',\alpha'_2,\alpha'_3\ldots$ ) the arrows of $Q(\tau)$ going around the marked point $p$ (resp. $q$). 

\medskip
Now, to the arc $i$ we associate five different  $\Lambda$-modules with top $S(i)$ the simple associated to  $i$:

\[\scalebox{0.8}{
\begin{tikzpicture}[>=stealth,scale=1]

\draw (4,0) node[below] {$i$};
\draw (3,0)--(5,0)--(4,1.5)--(3,0)--(4,-1.5)--(5,0);
\draw[thick, red, ->] (4,0) arc (0:120:1);
\node at (4,-2) {$A^L(i)$ (resp. $A^R(i)$)};

\draw (9,0) node[below] {$i$};
\draw (8,0)--(10,0)--(9,1.5)--(8,0)--(9,-1.5)--(10,0);
\draw[thick, red, ->] (9,0) arc (0:270:1);
\draw[thick, red, dotted, ->] (8,-1) arc (270:305:1);
\node at (9,-2) {$B^L(i)$ (resp. $B^R(i)$)};

\draw (1,-4) node[below] {$i$};
\draw (0,-4)--(2,-4)--(1,-2.5)--(0,-4)--(1,-5.5)--(2,-4);
\draw[thick, red, ->] (1,-4) arc (0:90:1);
\draw[thick, red,<-] (2,-5) arc (270:180:1);

\node at (1,-6) {$C(i)$};

\draw (7,-4) node[below] {$i$};
\draw (6,-4)--(8,-4)--(7,-2.5)--(6,-4)--(7,-5.5)--(8,-4);

\draw[thick, red, ->] (7,-4) arc (0:290:0.8);
\draw[thick, red,->] (7,-4)--(7.5,-4.75)--(6.5,-4.75);
\draw[thick, red,<-] (8,-4.8) arc (270:210:0.85);

\node at (7, -6) {$D^L(i)$ (resp. $D^R(i)$)};

\draw (13,-4) node[below] {$i$};
\draw (12,-4)--(14,-4)--(13,-2.5)--(12,-4)--(13,-5.5)--(14,-4);

\draw[thick, red, ->] (12.5,-3.25) arc (63:293:0.85);
\draw[thick, red,->] (13,-4)--(13.5,-4.75)--(12.5,-4.75);
\draw[thick, red, ->] (13,-4)--(12.5,-3.25)--(13.5,-3.25);
\draw[thick, red, ->] (13.5,-4.75) arc (-120:120:0.85);

\node at (13,-6) {$E(i)$};

\end{tikzpicture}}
\]

\begin{itemize}
\item
The module $A^L(i)$ exists if $\alpha$ is in $Q(\Lambda)$ and if the longest path $\alpha_\ell\ldots\alpha_2\alpha$ around the marked point $p$ in $Q(\Lambda)$ starting from $\alpha$  is not zero in $\Lambda$. It corresponds to the case where $p$ is on the boundary; or $p$ is a puncture and $\gamma'$ is in $\Lambda$; or $p$ is a puncture and $j'$ is an internal arc. In these cases, $A^L(i)$ is defined to be the string module associated to the string $\alpha_\ell\ldots\alpha_2\alpha$. The module $A^R(i)$ corresponds to the string $\alpha'_m\alpha'_2\alpha'$ around $q$. Note that the exponent $R$ and $L$ depends on the choice of the orientation on $i$.

\item The module $B^L(i)$ exists if $\alpha$ is in $Q(\Lambda)$. Then it is the string module associated to $\alpha_{\ell-1}\ldots\alpha_2\alpha$ where $\alpha_\ell\ldots\alpha_2\alpha$ is the longest path in $Q(\Lambda)$ around $p$. The module $B^R(i)$ is defined similarly around $q$.

\item The module $C(i)$ is defined if $\alpha$ and $\alpha'$ are in $Q(\Lambda)$. It is the string module associated to $\alpha_\ell\ldots\alpha_2\alpha\alpha'^{-1}\ldots\alpha'^{-1}_m$. Note that if $\alpha$ and $\alpha'$ are in $Q(\Lambda)$ then the paths $\alpha_\ell\ldots\alpha_2\alpha$ and $\alpha'_m\ldots\alpha'_2\alpha'$ are non zero in $\Lambda$.

\item The module $D^L(i)$ is defined if $p$ is a puncture, $j'$ is internal and $\alpha'$ and $\beta'$ are in $Q(\Lambda')$. Then $D^L(i)$ is the extension of $A^R(i)$ by $A^R(j)$ given by the arrows $\alpha$ and $\beta'$. 

\item The module $E(i)$ is defined if $p$ and $q$ are puncture and if $\alpha$, $\alpha'$, $\beta$ and $\beta'$ are arrows in $Q(\Lambda)$. Its socle is the semi-simple module $S(k)\oplus S(k')$ and it is the extension of $D^L(i)/S(k)$ by $S(k')$ given by the arrows $\alpha_\ell$ and $\beta'$.

\end{itemize}

\begin{proposition}\label{prop:shape-projectives}
Let $\tau$ be a valency $\geq 3$-triangualtion, $d$ an admissible cut and $\Lambda$ the degree zero subalgebra of $\mathcal{P}(Q(\tau),S(\tau),d)$. Let $i$ be an internal arc in $\tau$. Then the indecomposable projective module $P(i)$ associated to $i$ is either simple, or isomorphic to $A^{L,R}(i)$, $B^{L,R}(i)$, $C(i)$, $D^{L,R}(i)$ or $E(i)$.
\end{proposition}

\begin{proof}
Assume that $P(i)$ is not simple, then $i$ is not a sink in the quiver $Q(\Lambda)$. Hence we may assume that $\alpha$ is in $Q(\Lambda)$. Let $\alpha_\ell\ldots\alpha_2\alpha$ be the maximal composition of arrows around $p$ in the quiver $Q(\Lambda)$. If this composition vanishes in $\Lambda$ then $\alpha_\ell\ldots\alpha_2\alpha$ is in the Jacobian ideal and so $\gamma'$ exists, has degree $1$ and $\alpha_\ell\ldots\alpha_2\alpha=\partial_{\gamma'}S(\tau)$. 
This implies that $\Delta'$ is not internal. Indeed if $\Delta'$ would be internal, $\alpha'$ and $\beta'$ would exist and have degree zero, so we would have $\partial_{\gamma'}S(\tau)=\alpha_\ell\ldots\alpha_2\alpha-\beta'\alpha'$ in $\Lambda$. Therefore $j'$ is a boundary segment. Then if the hook $\beta\alpha$ exists in $Q(\Lambda)$, it vanishes since $q$ is on the boundary. Moreover for any hook $\beta_i\alpha_i$ the composition $\beta_i\alpha_i\alpha_{i-1}$ vanishes by Lemma \ref{lemma:hook=zero}. Hence there only one maximal path in $\Lambda$ starting f from $i$ which is $\alpha_{\ell-1}\ldots\alpha_2\alpha$, and $P(i)$ is isomorphic to the module $B^L(i)$.

Assume now that the path $\alpha_\ell\ldots\alpha_2\alpha$ does not vanish in $\Lambda$.  If the hook $\beta\alpha$ vanishes (or does not exist in $\Lambda$), then by the same argument as above, there exists only one maximal path starting with $\alpha$ in $\Lambda$. 

Assume that $\beta$ is in $Q(\Lambda)$ and $\beta\alpha\neq 0$ in $\Lambda$. Then $\Delta$ is an internal triangle, $q$ is a puncture and $d(\gamma)=1$. With similar arguments as above one deduces that there are only two maximal paths starting with $\alpha$ which are $\alpha_\ell\ldots\alpha$ and $\beta\alpha$. 

Combining these two remarks one deduces the following:

\begin{itemize}
\item if $\beta\alpha=0$ (or does not exist) in $\Lambda$ and $\alpha'$ is not in $Q(\Lambda)$ then $P(i)$ is isormophic to $A^L(i)$;
\item if $\beta\alpha=0$ (or does not exist) in $\Lambda$, $\alpha'$ is in $Q(\Lambda)$ and $\beta'\alpha'=0$ (or does not exist) then $P(i)$ is isomorphic to $C(i)$;
\item if $\beta\alpha=0$ (or does not exist) in $\Lambda$, $\alpha'$ is in $Q(\Lambda)$ and $\beta'\alpha'\neq 0$ then $P(i)$ is isomorphic to $D^R(i)$;
\item if $\beta\alpha\neq 0$, $\alpha'$ is in $Q(\Lambda)$ and $\beta'\alpha'=0$ (or does not exist) then $P(i)$ is isomorphic to $D^L(i)$;
\item if $\beta\alpha\neq 0$ and $\beta'\alpha'\neq 0$ in $\Lambda$, then $P(i)$ is isomorphic to $E(i)$.
\end{itemize}

\end{proof}

\subsection{Projective dimension of simples}
 In this subsection, we adopt the following notation. By a black bullet we denote a marked point that is a puncture, and by a gray square a marked point that is on the boundary. A side of a triangle that is a boundary segment  is coloured in gray and an angle (between two internal arcs) corresponding to an arrow of $Q(\tau)$ of degree $1$ is denoted by a double angle.

\begin{proposition}\label{prop:pd(S_i)>2}
Let $\tau$, $d$, $\Lambda$ and $i$ as in Proposition \ref{prop:shape-projectives}. We assume that the orientations of the arrows around a marked point is counterclockwise. Then the simple $\Lambda$-module $S=S(i)$ has projective dimension $\geq 3$ if and only $i$ lies in one of the following configurations:

\[\scalebox{0.8}{
\begin{tikzpicture}[>=stealth,scale=1]
\draw (1,0) node[below] {$i$};
\draw (-1,1.5)--(1,1.5)--(0,0)--(2,0)--(1,1.5);
\draw [gray, line width=5pt] (-1,1.5)--(0,0);
\node at  (1,1.5) {$\bullet$};
\draw [thick] (1.75,0.36) arc (120:180:0.4);
\draw [thick] (0.6,1.5) arc (180:240:0.4);

\draw (7,0) node[below] {$i$};
\draw (6,0)--(8,0)--(7,1.5)--(9,1.5)--(8,0);
\draw [gray, line width=5pt] (6,0)--(7,1.5);
\node at (8,0) {$\bullet$};
\draw [thick] (7.75,0.35) arc (120: 180:0.4);
\draw[thick] (8.6, 1.5) arc (180:240:0.4);

\end{tikzpicture}}\]

\end{proposition}

\begin{proof}
For a module $M$ we denote by $P(M)$ its projective cover, and for an arc $\ell$ we denote $P(\ell)$ the projective cover of the simple $S(\ell)$. 

If $P(i)$ is simple, then the projective dimension of $S(i)$ is $0$. 

\medskip
\textbf{Case 1: $P(i)\simeq A^L(i)$.}

\noindent
Then the first syzygy $\Omega S$ of $S$ is isomorphic to $S(j)$ or to $A^R(j)$ depending wether $\alpha_2$ is in $Q(\Lambda)$ or not. Hence we have $P(\Omega S)=P(j)$. 
Moreover since $P(i)=A^L(i)$, the arrow $\alpha'$ is not in $Q(\Lambda)$ and either $\beta$ is not in $Q(\Lambda)$ or $\beta\alpha=0$ in $\Lambda$. We distinguish the two cases.

\begin{itemize}
\item \textit{$\beta$ is not $Q(\Lambda)$:} In that case, $A^R(j)$ (resp. $S(j)$ if $\alpha_2$ is not in $Q(\Lambda)$) is projective, hence $S$ has projective dimension $1$.
\item \textit{$\beta$ is in $Q(\Lambda)$ and $\beta\alpha=0$:} In that case, since $d(\gamma)=1$, the module $A^R(k)$ (resp. $S(k)$ if $\beta_2$ is not in $Q(\Lambda)$) is projective. If $\Omega S=A^R(j)$, then there are two arrows starting from $j$ in $Q(\Lambda)$. So $P(j)$ is isomorphic to $C(j)$, $D^{L,R}(j)$ or $E(j)$. It cannot be $D^R(j)$ and $E(j)$ since $\alpha$ is in $Q(\Lambda)$. In the other two cases we have $\Omega^2S\simeq A^R(k)$ (resp. $S(k)$ if $\beta_2$ is not in $Q(\Lambda)$) which is projective. Thus $S$ has projective dimension two.

If $\Omega S=S(j)$, then there is only one arrow in $Q(\Lambda)$ starting from $j$, so $P(j)$ is of type $A^L(j)$ or
$B^L(j)$. If it is $A^L(j)$ then $\Omega^2S$ is isomorphic to $A^R(k)$ (or $S(k)$) and the projective dimension of $S$ is two.

If $P(j)=B^L(j)$, then $\Omega^2S$ is isomorphic to $B^L(k)$ which is not projective. So the projective dimension of $S$ is $\geq 3$, and $i$ lies in one of the following configurations:

\[\scalebox{1}{
\begin{tikzpicture}[>=stealth,scale=0.8]
\draw (1,0) node[below] {$i$};
\draw (0,0)--(2,0)--(1,1.5)--(0,0)--(1,-1.5)--(2,0)--(1,1.5)--(-1,1.5);
\draw [gray, line width=5pt] (-1,1.5)--(0,0);
\node at  (1,1.5) {$\bullet$};

\draw [thick] (1.75,0.36) arc (120:240:0.4);
\node at (2,0) [gray] {$\blacksquare$};
\draw [thick] (0.6,1.5) arc (180:240:0.4);

\draw (7,0) node[below] {$i$};
\draw (6,0)--(8,0)--(7,1.5)--(6,0)--(7,-1.5)--(8,0)--(7,1.5)--(5,1.5);
\draw [gray, line width=5pt] (5,1.5)--(6,0);
\draw [gray, line width=5pt] (7,-1.5)--(8,0);
\node at  (7,1.5) {$\bullet$};
\draw [thick] (7.75,0.36) arc (120:180:0.4);

\draw [thick] (6.6,1.5) arc (180:240:0.4);

\draw (13,0) node[below] {$i$};
\draw (12,0)--(14,0)--(13,1.5)--(12,0)--(13,-1.5)--(14,0)--(13,1.5)--(11,1.5);
\draw [gray, line width=5pt] (11,1.5)--(13,-1.5);
\draw [gray, line width=5pt] (13,-1.5)--(14,0);

\node at  (13,1.5) {$\bullet$};
\draw [thick] (13.75,0.36) arc (120:180:0.4);
\draw [thick] (12.6,1.5) arc (180:240:0.4);

\end{tikzpicture}}\]

In the case where $P(i)\simeq A^R(i)$, then we obtain the same configurations (up to central symmetry).
 
\end{itemize}

\medskip
\textbf{Case 2: $P(i)\simeq B^L(i)$.} 

\noindent
In this case, we have $\Omega S\simeq B^R(j)$ (note that the module $B^R(j)$ might be simple), hence the projective cover of $\Omega S$ is $P(j)$. Again, we seperate the cases depending on the existence of $\beta$ in $Q(\Lambda)$.

\begin{itemize}

\item \textit{$\beta$ is not in $Q(\Lambda)$:} Then we have $P(j)=A^R(j)$ and $\Omega^2S=S(k')$. 

\item \textit{$\beta$ is in $Q(\Lambda)$:} Then $\beta\alpha$ vanishes in $\Lambda$ since $q$ is on the boundary. By the same argument as in Case 1, the projective $P(j)$ is isomorphic to $B(j)$ or $D^L(j)$ (or to $A^L(j)$ if $\Omega S=S(j)$), and $\Omega^2S\simeq S(k')\oplus A^R(k)$ ( or $S(k')\oplus S(k)$ if $k$ is a sink in $Q(\Lambda)$). Moreover, the module $A^R(k)$ is always projective. 
\end{itemize} 

Hence, in this situation, $S$ has projective dimension $\geq 3$ if and only if the simple $S(k')$ is not projective, i.e. if $i$ lies in the following configuration:

\[\scalebox{1}{
\begin{tikzpicture}[>=stealth,scale=0.8]
\draw (1,0) node[above] {$i$};
\draw (0,0)--(1,-1.5)--(-1,-1.5)--(0,0)--(2,0);
\draw [gray, line width=5pt] (1,-1.5)--(2,0);
\node at  (0,0) {$\bullet$};

\draw [thick] (0.4,0) arc (0:-60:0.4);
\draw [thick] (-0.6,-1.5) arc (0:60:0.4);

\end{tikzpicture}}\]

\medskip

\textbf{Case 3: $P(i)\simeq C(i)$.}

In this case $\Omega S$ is isomorphic to $A^R(j)\oplus A^R(j')$, (or $A^R(j)\oplus S(j')$ or $S(j)\oplus A^R(j')$ or $S(j)\oplus S(j')$ depending wether $\alpha_2$ and $\alpha'_2$ are in $Q(\Lambda)$). The situation is then the same as in Case 1, that is $S$ has projective dimension $\geq 3$ if and only if $P(j)$ (or $P(j')$) are isomorphic to $B^L(j)$ (or $B^L(j')$). This corresponds to the following situations (up to central symmetry):  

\[\scalebox{1}{
\begin{tikzpicture}[>=stealth,scale=0.8]
\draw (1,0) node[below] {$i$};
\draw (0,0)--(2,0)--(1,1.5)--(0,0)--(1,-1.5)--(2,0)--(1,1.5)--(-1,1.5);
\draw [gray, line width=5pt] (-1,1.5)--(1,-1.5);
\node at  (1,1.5) {$\bullet$};

\draw [thick] (1.6,0) arc (180:120:0.4);
\node at (2,0) [gray] {$\blacksquare$};
\draw [thick] (0.6,1.5) arc (180:240:0.4);

\draw (7,0) node[below] {$i$};
\draw (6,0)--(8,0)--(7,1.5)--(6,0)--(7,-1.5)--(8,0)--(7,1.5)--(5,1.5);
\draw [gray, line width=5pt] (5,1.5)--(6,0);
\node at (8,0) [gray] {$\blacksquare$};
\node at  (7,1.5) {$\bullet$};
\draw [thick] (7.75,0.36) arc (120:180:0.4);
\draw [thick] (6.6,1.5) arc (180:240:0.4);
\draw[thick] (6.8,-1.2) arc (120:60:0.4);

\draw (13,0) node[below] {$i$};
\draw (12,0)--(14,0)--(13,1.5)--(12,0)--(13,-1.5)--(14,0)--(13,1.5)--(11,1.5);
\draw [gray, line width=5pt] (11,1.5)--(12,0);
\node at (14,0) [gray] {$\blacksquare$};
\node at  (13,1.5) {$\bullet$};
\draw [thick] (13.75,0.36) arc (120:180:0.4);
\draw [thick] (12.6,1.5) arc (180:240:0.4);
\draw[thick] (12.4,0) arc (0:-60:0.4);

\end{tikzpicture}}\]

\medskip

\textbf{Case 4: $P(i)\simeq D^L(i)$.}

\noindent
Then the module $\Omega S$ has the following shape:

\[\scalebox{0.8}{
\begin{tikzpicture}[>=stealth,scale=1]
 \draw (1,0) node[below] {$i$};
 \node at (0.5,0.75) [right]{$j$};
 
 \node at (0,0) {$\bullet$};
 \draw [thick] (0.4,0) arc (0:-60:0.4);
 \node at (1.5,-0.75) [right] {$j'$};
\draw (0,0)--(2,0)--(1,1.5)--(0,0)--(1,-1.5)--(2,0);
\draw[thick, red, ->] (0.5,0.75) arc (60:300:0.9);
\draw[thick, red,->] (1.5,-0.75)--(0.5,-0.75);
\draw[thick, red,->] (1.5,-0.75) arc (-120:-30:1);

\end{tikzpicture}}\]

The top of this module is $S(j)\oplus S(j')$, so we have $P(\Omega S)\simeq P(j)\oplus P(j')$. 

\begin{itemize}
\item \textit{$\beta$ is not in $Q(\Lambda)$:} Then the projective $P(j)$ is isomorphic to $A^R(j)$. Indeed it is not simple since each puncture has valency at least $3$ and it cannot be $B^R(j)$ since $\alpha$ is in $Q(\Lambda)$.
The projective $P(j')$ can be isomorphic to $A^L(j')$, $B^L(j')$, $C(j')$ or $D^L(j')$. If $P(j')\neq B^L(j')$ then $\Omega^2S$ is isomorphic to $A^R(k')$ (or $S(k')$) which is projective. If $P(j')\simeq B^L(j')$ then $\Omega^2 S$ is isomorphic to $B^R(k')$ which is not projective.   
  
\item \textit{$\beta$ is in $Q(\Lambda)$:} Then the hook $\beta\alpha$ vanishes in $\Lambda$. The projective $P(j)$ is isomorphic to $C(j)$ or $D^L(j)$, and the projective $P(j')$ is isomorphic to $A^L(j')$, $B^L(j')$, $C(j')$ or $D^L(j')$. As in the previous case, if $P(j')$ is not isomorphic to $B^L(j')$ then $\Omega^2 S$ is isomorphic to $A^R(k)\oplus A^R(k')$ which are both projective. Hence $S$ has projective dimension $\geq 3$ if and only if $P(j')\simeq B^L(j')$.

\end{itemize}

Therefore, if $P(i)\simeq D^L(i)$, the simple $S$ has projective dimension $\geq 3$ if and only if $i$ lies in the following configuration:

\[\scalebox{0.8}{
\begin{tikzpicture}[>=stealth,scale=1]
\draw (1,0) node[below] {$i$};
\draw (2,0)--(0,0)--(1,-1.5)--(2,0)--(3,-1.5)--(1,-1.5);
\draw [gray, line width=5pt] (2,0)--(3,-1.5);
\node at  (0,0) {$\bullet$};
\draw [thick] (.25,-0.36) arc (300:360:0.4);
\node at (1,-1.5)  {$\bullet$};
\draw [thick] (1.4,-1.5) arc (0:60:0.4);

\end{tikzpicture}}\]

\medskip

\textbf{Case 5: $P(i)\simeq E(i)$.}

\noindent
Then  $\Omega S$ has the following form.
\[\scalebox{0.8}{
\begin{tikzpicture}[>=stealth,scale=1]
 \draw (1,0) node[below] {$i$};
 \node at (0.5,0.75) [above]{$j$};
 \node at (1.5,-0.75) [below] {$j'$};
\draw (0,0)--(2,0)--(1,1.5)--(0,0)--(1,-1.5)--(2,0);
\draw[thick, red, ->] (0.5,0.75) arc (60:300:0.85);
\draw[thick, red,->] (1.5,-0.75)--(0.5,-0.75);
\draw[thick, red,->] (1.5,-0.75) arc (-120:120:0.85);
\draw [thick, red, <-] (1.5,0.75)--(0.5,0.75);
\node at (0,0) {$\bullet$};
\node at (2,0) {$\bullet$};
 \draw [thick] (0.4,0) arc (0:-60:0.4);
  \draw [thick] (1.6,0) arc (180:120:0.4);

\end{tikzpicture}}\]

The top of this module is $S(j)\oplus S(j')$ so $P(\Omega S)\simeq P(j)\oplus P(j')$. Then as in the previous case, the projective $P(j)$ (resp. $P(j')$) is isomorphic to $B(j)$ or to $D^L(j)$ (resp. $B(j')$ or $D^L(j')$). In all these cases we get $\Omega^2S\simeq A^R(k)\oplus A^R(k')$ and is projective.

\end{proof}

We end this section with the following result, which implies that for any surface $\surf$ with boundary there exists an associated surface cut algebra.

\begin{corollary}\label{cor:existence-surface-algebra} Let $\surf$ be a surface with non empty boundary which is not:
\begin{itemize}
\item a $n$-gon with $n\leq 3$;
\item a once punctured monogon or digon;
\item a twice punctured monogon.
\end{itemize}
Then there exists a valency $\geq 3$-triangulation $\tau$ and an admissible cut $d$ such that the corresponding cut algebra $\Lambda$ has global dimension $\leq 2$.
\end{corollary}

\begin{proof}
We proceed by induction on the number of punctures. Assume that there exists a triangulation $\tau$ and a cut $d$ with the following properties:
\begin{enumerate}
\item $\tau$ is a valency $\geq 3$-triangulation;
\item there exists a triangle $\Delta$ in $\tau$ with exactly one side being a boundary segment;
\item for each triangle of $\tau$ containing one puncture and one boundary segment, the corresponding arrow has degree $0$.
\end{enumerate}

By Proposition \ref{prop:pd(S_i)>2}, such graded triangulation gives rise to a  cut algebra of global dimension $\leq 2$. 

Now we add a puncture $p$ in the surface, and we construct a graded triangulation $(\tau',d')$ of the surface $(\Sigma,\mathbb M\cup\{p\})$ with the same properties.  We may assume that $p$ is in the triangle $\Delta$. Denote by $i$ and $j$ the sides of $\Delta$ that are arcs in $\tau$, and link $p$ to the three vertices of $\Delta$ with arcs $\ell$, $\ell_i$ and $\ell_j$. 
  \[\scalebox{0.8}{
\begin{tikzpicture}[>=stealth,scale=1]

\draw (2,3)--node[fill=white,inner sep=0pt]{$j$}(4,0)--node[fill=white,inner sep=0pt, left]{$\ell_j$}(2,1)--node[fill=white,inner sep=1pt]{$\ell$}(2,3)--node[fill=white,inner sep=0pt]{$i$}(0,0)--node[fill=white,inner sep=0pt]{$\ell_i$}(2,1);
\node at (2,1) {$\bullet$};
\draw[thick] (2,1.4) arc (90:208:0.4);
\draw[thick] (3.5,0.75) arc (120:150:1);
\node at (1,1) {$\Delta_i$};
\node at (3,1) {$\Delta_j$};
\node at (2,0.5) {$\Delta'$};
 
\draw [gray, line width=5pt] (0,0)--(4,0);

\end{tikzpicture}}\]

We obtain a valency $\geq 3$-triangulation $\tau'=\tau\cup \{\ell,\ell_i,\ell_j\}$ of $(\Sigma,\mathbb M\cup\{p\})$. This triangulation contains three new triangles: two internal $\Delta_i$ (resp. $\Delta_j$) formed respectively by the arcs $i$, $\ell_{i}$ and $\ell$ (resp.  $j$, $\ell_j$ and $\ell$) and one triangle $\Delta'$ with exactly one side being a boundary segment. Now define the following cut $d'$ on $\tau'$:
\[d'(\alpha) = \left\{\begin{array}{ll}  d(\alpha) & \textrm{if } \alpha\in Q(\tau)\cap Q(\tau');\\
  1 & \textrm{if }\alpha\textrm{ is the arrow in }\Delta_i\textrm{ parallel to }i;\\ 
  1 & \textrm{if }\alpha\textrm{ is the arrow in }\Delta_j\textrm{ parallel to }\ell;\\
  0 & \textrm{else.} 
  \end{array}\right.\]
  It is straigtforward to check that $d'$ is a cut for $\tau'$ satisfying $(3)$ above.

 Now there clearly exists $(\tau,d)$ satisfying $(1)$, $(2)$ and $(3)$ in case $\surf$ has no punctures and is not a $n$-gon with $n\leq 4$. Moreover the next picture shows the existence of such a $(\tau,d)$ in case $\surf$ is a once punctured $4$-gon, twice punctured $2$ or $3$-gon, and a monogon with three punctures. 
 
   \[\scalebox{0.8}{
\begin{tikzpicture}[>=stealth,scale=1]
\draw[gray, line width=5pt] (0,0) circle (2);
\node at (0,0){$\bullet$};
\node at (-2,0){$\bullet$};
\node at (2,0){$\bullet$};
\node at (0,2){$\bullet$};
\node at (0,-2){$\bullet$};
\draw (2,0)--(0,0)--(0,2)..controls (-1.5,1.5) and (-1.5,-1.5)..(0,-2)--(0,0);
\node at (1,1) {$\Delta$};
\draw[thick] (0,0.3) arc (90:270:0.3);

\draw [gray, line width=5pt](6,0) circle (2);
\node at (4,0){$\bullet$};
\node at (5,0){$\bullet$};
\node at (7,0){$\bullet$};
\node at (6,2){$\bullet$};
\node at (6,-2){$\bullet$};
\draw (5,0)--(6,2)--(7,0)--(5,0)--(6,-2)--(7,0);
\draw[dotted] (4,0)--(5,0);
\node at (7.5,0) {$\Delta$};
\draw[thick] (5.5,0) arc (0:62:0.5);
\draw[thick] (6.5,0) arc (180:242:0.5);

\draw [gray, line width=5pt](12,0) circle (2);
\node at (12,2){$\bullet$};
\node at (12,-1){$\bullet$};
\node at (11.5,0){$\bullet$};
\node at (12.5,0){$\bullet$};
\draw (11.5,0)--(12,2)--(12.5,0)--(11.5,0)--(12,-1)--(12.5,0);
\draw (12,2).. controls (10.5,1) and (11,-1)..(12,-1);
\draw (12,2).. controls (13.5,1) and (13,-1)..(12,-1);
\draw[thick] (11.5,1.55) arc (190:255:0.4);
\draw[thick] (11.8,0) arc (0:70:0.3);
\draw[thick] (12.2,0) arc (180:250:0.3);
\draw[thick] (12.1,-0.75) arc (60:10:0.3);
\node at (12,-1.5) {$\Delta$};
\end{tikzpicture}}\]

The only remaining case is the case of a once punctured triangle. In this case, there does not exist a graded triangulation with properties $(1)$, $(2)$ and $(3)$ but there clearly exists a valency $\leq 3$-triangulation and a cut such that the cut algebra has global dimension $\leq 2$.
\end{proof}

\section*{Acknowledgements}

A few of the results in this work were obtained during a visit that the second author made to the first and third authors in June 2014. The visit was made possible by the CNRS-CONACyT's Solomon Lefschetz International Laboratory (LAISLA), whose financial support we gratefully acknowledge.


\begin{thebibliography}{50}

\bibitem{Amiot-gldim2} C. Amiot. {\it Cluster categories for algebras of global dimension 2 and quivers with potential}. Annales de l'Institut Fourier, 59 no. 6 (2009), 2525-2590. arXiv:0805.1035

\bibitem{Amiot-survey} C. Amiot. {\it On generalized cluster categories}.  ICRA XIV Proceedings, Eur. Math. Soc. (2011). arXiv:1101.3675

\bibitem{Amiot-Grimeland} C. Amiot, Y. Grimeland. {\it Derived invariants for surface algebras}. Journal of Pure and Applied Algebra 220 (2016), pp. 3133-3155arXiv1411.0383.

\bibitem{Amiot-genus1} C. Amiot {\it The derived category of surface algebras: the case of the torus with one boundary component.} Algebras and Representation Theory, (2016), 1-22, arXiv 1506.02410.

\bibitem{AO-cl-equiv-der-equiv} C. Amiot, S. Oppermann. {\it Cluster equivalence and graded derived equivalence}. Documenta Math. \textbf{19} (2014) 1155--1206, arXiv:1003.4916

\bibitem{AO-acyc-cl-type} C. Amiot, S. Oppermann. {\it Algebras of acyclic cluster type: tree type and type $\widetilde{A}$}. Nagoya Math. J., 211: 1-50, 2013. arXiv:1009.4065

\bibitem{ABCP} I. Assem, T. Br\"{u}stle, G. Charbonneau-Jodoin, P-G. Plamondon. {\it Gentle algebras arising from surface triangulations}. Algebra \& Number Theory Vol. 4 (2010), No. 2, 201-229. arXiv:0903.3347

\bibitem{Assem-Brustle-Schiffler} I. Assem, T. Br\"ustle and R. Schiffler, {\it Cluster-tilted algebras as trivial extensions}, Bull. Lond. Math. Soc. 40 (2008), no. 1, 151--162. 

\bibitem{Bridgeland-Smith} T. Bridgeland and I. Smith, {\it Quadratic differentials as stability conditions}.  arXiv:1302.7030 [math.AG].


\bibitem{Broomhead-paper} N. Broomhead. {\it Dimer models and Calabi-Yau algebras}. Memoirs of the American Mathematical Society, Vol. 215, No. 1011. 2012.

\bibitem{Brustle-Qiu} T. Br\"{u}stle, Y. Qiu. {\it Tagged mapping class groups I: Auslander-Reiten translation}. arXiv:1212.0007

\bibitem{Brustle-Zhang} T. Br\"{u}stle, J. Zhang. {\it On the Cluster Category of a Marked Surface}. Algebra and Number Theory, 5 (2011), No. 4, 529-566. arXiv:1005.2422

\bibitem{Caldero-Chapoton} P. Caldero and F. Chapoton, {\it Cluster algebras as Hall algebras of quiver representations}, Comment. Math. Helv. 81 (2006), no. 3, 595--616. 

\bibitem{Caldero-Chapoton-Schiffler} P. Caldero, F. Chapoton and R. Schiffler, {\it Quivers with relations arising from clusters (An case)}, Trans. Amer. Math. Soc. 358 (2006), no. 3, 1347--1364. 

\bibitem{CI-LF} G. Cerulli Irelli, D. Labardini-Fragoso. {\it Quivers with potentials associated to triangulated surfaces, part III: Tagged triangulations and cluster monomials}. \textsl{Compositio Mathematica} \textbf{148} (2012), No. 06,  1833--1866. arXiv:1108.1774

\bibitem{CKLP} G. Cerulli Irelli, B. Keller, D. Labardini-Fragoso, P-G. Plamondon. {\it Linear independence of cluster monomials for skew-symmetric cluster algebras}. \textsl{Compositio Mathematica} \textbf{149} (2013), No. 10, 1753--1764.

\bibitem{Cohen} M. M. Cohen. {\it A course in simple-homotopy theory}. Graduate Texts in Mathematics 10, Springer-Verlag. 1973.

\bibitem{Cooke-Finney} G. E. Cooke, R. L. Finney. {\it Homology of cell complexes} (based on Lectures by N. E. Steenrod). Mathematical Notes, Princeton University Press and the University of Tokyo Press, 1967.

\bibitem{David-Roesler} L. David-Roesler. {\it Derived Equivalence of Surface Algebras in Genus 0 via Graded Equivalence}. Algebras and Representation Theory
Vol. 17 (2014), No. 1, 1-30. arXiv:1111.4657

\bibitem{DavidR-Schiffler} L. David-Roesler and R. Schiffler. {\it Algebras from surfaces without punctures}. J. Algebra 350 (2012), No. 1, 218-244. arXiv:1103.4357

\bibitem{DWZ1} H. Derksen, J. Weyman, A. Zelevinsky. {\it Quivers with potentials and their representations I: Mutations}. Selecta Math. 14 (2008), no. 1, 59�119. arXiv:0704.0649

\bibitem{FST} S. Fomin, M. Shapiro, D. Thurston. {\it Cluster algebras and triangulated surfaces, part I: Cluster complexes}. Acta Mathematica 201 (2008), 83-146. arXiv:math.RA/0608367

\bibitem{Fomin-Zelevinsky} S. Fomin and A. Zelevinsky, {\it Cluster algebras. I. Foundations.} J. Amer. Math. Soc. {\bf 15} (2002), no. 2, 497--529. 

\bibitem{GG}
R. Gordon and E. L. Green.
 {\it Representation theory of graded {A}rtin algebras}. J. Algebra, \textbf{76} (1):138--152, 1982.


\bibitem{Gu1} W. Gu, {\it A Decomposition Algorithm for the Oriented Adjacency Graph of the Triangulations of a Bordered Surface with Marked Points}, The Electronic Journal of Combinatorics {\bf 18} (2011), \#P91.

\bibitem{Gu2} W. Gu, {\it Graphs with Non-unique Decomposition and Their Associated Surfaces}, arXiv:1112.1008 [math.CO].

\bibitem{Hall} P. Hall. {\it On Representatives of Subsets}. J. London Math. Soc. (1935) s1-10 (1): 26--30. doi: 10.1112/jlms/s1-10.37.26

\bibitem{Happel} D. Happel. {\it Triangulated categories in the representation theory of finite-dimensional algebras}. London Mathematical Society Lecture Note Series, vol. 119. Cambridge University Press. Cambridge, 1988.

\bibitem{Keller-orbit-cats} B. Keller. {\it On triangulated orbit categories}. Doc. Math., 10: 551-581, 2005.

\bibitem{Keller-CY-compl} B. Keller. {\it Deformed Calabi-Yau completions}. J. Reine Angew. Math., 654:125-180, 2011. With an appendix by Michel Van den Bergh.

\bibitem{Keller-Yang} B. Keller, D. Yang. {\it Derived equivalences from mutations of quivers with potential}. Advances in Mathematics 226 (2011), 2118-2168. arXiv:0906.0761

\bibitem{LF1} D. Labardini-Fragoso. {\it Quivers with potentials associated to triangulated surfaces}. Proc. London Mathematical Society (2009) 98 (3): 797-839. arXiv:0803.1328


\bibitem{LF2} D. Labardini-Fragoso. {\it Quivers with potentials associated to triangulated surfaces, part II: Arc representations}. arXiv: 0909.4100

\bibitem{LF4} D. Labardini-Fragoso. {\it Quivers with potentials associated to triangulated surfaces, part IV: Removing boundary assumptions}. arXiv:1206.1798

\bibitem{Ladkani} S. Ladkani. {\it On Jacobian algebras from closed surfaces}. arXiv: 1207.3778

\bibitem{Mozgovoy-Reineke}S. Mozgovoy, M. Reineke. {\it On the noncommutative Donaldson-Thomas invariants arising from brane tilings}. Advances in Mathematics 223 (2010), 1521--1544. arXiv:0809.0117

\end{thebibliography}
\end{document}